\documentclass{article}

\usepackage{PRIMEarxiv}

\usepackage[utf8]{inputenc} 
\usepackage[T1]{fontenc}    
\usepackage{hyperref}       
\usepackage{url}            
\usepackage{booktabs}       
\usepackage{amsfonts}       
\usepackage{nicefrac}       
\usepackage{microtype}      
\usepackage{lipsum}
\usepackage{fancyhdr}       
\usepackage{graphicx}       
\graphicspath{{media/}}     
\usepackage{amssymb}
\usepackage{amsmath}
\usepackage{physics}
\usepackage{subcaption}
\usepackage{xcolor}
 \usepackage{amsthm}
\newcommand*{\ngradpar}{\nabla_{\parallel}}
\newcommand*{\ngradperp}{\nabla_{\perp}}
\newcommand*{\nlaplpar}{\nabla_{\parallel}^2}
\newcommand*{\nlaplperp}{\nabla_{\perp}^2}
\newcommand*{\vpare}{V_{\parallel e}}
\newcommand*{\ijkv}{i^+,j^+,k^+}
\newcommand*{\xijkv}{x_{i^+},y_{j^+},z_{k^+}}

\newtheorem{thm}{Theorem}
\newtheorem{rmk}{Remark}
\newtheorem{corollary}{Corollary}[thm]
\title{Mimetic finite difference schemes for transport operators with divergence--free advective field\\and applications to plasma physics
}

\author{
  Micol Bassanini \\
   Institute of Mathematics \& Swiss Plasma Center (SPC)\\
  École Polytechnique Fédérale de Lausanne (EPFL)\\
  Lausanne, 1015, Switzerland \\
  \texttt{micol.bassanini@epfl.ch} \\
  \And 
  Simone Deparis\\
  Institute of Mathematics \\
  École Polytechnique Fédérale de Lausanne (EPFL)\\
  Lausanne, 1015, Switzerland \\
   \And
  Paolo Ricci \\
  Swiss Plasma Center (SPC)\\
  École Polytechnique Fédérale de Lausanne (EPFL)\\
  Lausanne, 1015, Switzerland \\
}

\begin{document}
\maketitle

\begin{abstract}
In wave propagation problems, finite difference methods implemented on staggered grids are commonly used to avoid checkerboard patterns and to improve accuracy in the approximation of short--wavelength components of the solutions.  
In this study, we develop a mimetic finite difference (MFD) method on staggered grids for transport operators with divergence--free advective field that is proven to be energy--preserving in wave problems. This method mimics some characteristics of the summation--by--parts (SBP) operators framework, in particular it preserves the divergence theorem at the discrete level. Its design is intended to be versatile and applicable to wave problems characterized by a divergence--free velocity.
As an application, we consider the electrostatic shear Alfvén waves (SAWs), appearing in the modeling of plasmas. These waves are solved in a magnetic field configuration recalling that of a tokamak device. The study of the generalized eigenvalue problem associated with the SAWs shows the energy conservation of the discretization scheme, demonstrating the stability of the numerical solution.

\end{abstract}


\section{Introduction}
\label{sec1}

The use of staggered grids is a common technique in addressing wave propagation challenges, offering a direct and effective approach to avoid odd--even decoupling issues, \cite{patankar2018numerical}. Odd--even decoupling is a common numerical error that appears in collocated grids, where all variables are stored at identical locations, resulting in checkerboard patterns of the solution, \cite{durran2010numerical}. Moreover, the staggered grids provide a more accurate approximation of both the phase and the group speed in the case of short--wavelength components comparable with the grid size \cite{durran2013numerical}.
Examples of the use of the staggered grid can be found in computational fluid dynamics, as one of the first strategies to avoid pressure--velocity decoupling, in the context of solving Navier--Stokes equations through finite difference methods \cite{sharma2021introduction}, \cite{patankar2018numerical}. Within this family of numerical methods, the Yee scheme holds particular significance, offering a reliable methodology for discretizing and solving Maxwell's equations by employing staggered grids in both temporal and spatial dimensions, \cite{yee1966numerical}.

 It has been consistently observed that the reliability of numerical simulations significantly improves when the numerical discretization preserves or mimics the fundamental mathematical properties of the physical model, \cite{lipnikov2014mimetic}. Indeed, the mimetic finite difference (MFD) methods are designed to preserve key properties of continuum equations, such as energy conservation, at the discrete level \cite{lipnikov2014mimetic}. For example, in the works by Arakawa and Lamb, in \cite{arakawa1981potential} and \cite{arakawa1977computational},
they identified a numerical scheme on staggered grids capable of conserving potential enstrophy and total energy for the flow of the shallow water equations. 
Summation--by--parts (SBP) operators are an example of this family, designed to replicate integration by parts at the discrete level \cite{fernandez2014review}, \cite{svard2014review}. 

In the present work, we propose a MFD scheme, based on the use of the staggered grid with characteristics typical of SBP operators and able to preserve the divergence theorem at the discrete level. 
In its standard form, SBP operators are defined in collocation grids, where all variables are stored at the same grid
points, for handling first or second derivatives. In contrast, the present method discretizes advective operators with divergence--free velocity on a staggered grid.
While previous research has explored the extension of SBP methods to staggered grids for wave propagation problems, as discussed in \cite{o2017energy}, the approach we present specifically addresses transport operators with divergence--free advective fields.
This approach is crucial for preserving the energy conservation properties of a wave problem. In this work, we do not focus on the imposition of boundary conditions. This task presents an additional challenge for high--order finite difference schemes, as solutions in different parts of the domain must be accurately and stably connected. The stencils near boundaries introduce further complexities. One approach to address this issue is the Simultaneous--Approximation--Term (SAT) technique, which applies boundary and interface conditions in a weak form, \cite{svard2014review}.

The advantage of the numerical scheme we propose is of particular interest for systems that present a strong anisotropy. This is the case of strongly magnetized plasma.
Indeed, the space scale along the direction parallel to the equilibrium magnetic field is orders of magnitude greater than the scale length perpendicular to the magnetic field, making the discretization of the parallel gradient $\ngradpar f=\textbf{b}\cdot \boldsymbol{\nabla}f$, where $\textbf{b}$ denotes the unit vector of the magnetic field, particularly challenging.
Anisotropy is frequently addressed by using coordinates aligned to $\textbf{b}$, which allows reducing the numerical grid density along the resultant parallel direction. This strategy is particularly effective for modeling the core region of fusion devices \cite{garbet2010gyrokinetic}. 
However, field-aligned coordinates encounter singularity issues in e.g. the simulation of fusion devices \cite{stegmeir2023analysis}. Singularities in field--aligned coordinate systems can arise when the magnetic field configuration includes saddle points, as is common in diverted geometries. 
One straightforward strategy to address this issue, used for example in the BOUT code \cite{umansky2009status}, is to avoid placing grid points at the magnetic saddle and to partition the domain into subregions where field--aligned coordinates remain well--defined. 
An alternative strategy, implemented in the GBS code \cite{giacomin2022gbs}, is to discretize the equations on a grid that, in the limit of large aspect ratio, is Cartesian and uniform \cite{ricci2012simulation}. Unlike field--aligned coordinate systems, this approach is independent of the equilibrium magnetic field structure. 
Furthermore, this approach facilitates the implementation of boundary conditions representing plasma--wall interactions, as the domain boundaries align naturally with the physical geometry of the device. Recently, advancements have been made to incorporate a curvilinear finite difference scheme into GBS, enabling the simulation of more complex geometries. This enhancement provides greater flexibility in representing fixed wall boundaries without having the issues related to the field--aligned coordinates.
The straightforward discretization of the parallel gradient using non--aligned coordinates and staggered grids does not preserve the divergence theorem at the discrete level.
Significant efforts have been made to discretize the parallel Laplacian operator $\nlaplpar$ using finite difference methods  in the study of high magnetized plasma \cite{GUNTER2005354}, \cite{gunter2007finite} exploiting a grid staggered with respect to the original one in all the directions. We prove that in a particular case an approach leads to the implementation of the parallel Laplacian reported in \cite{GUNTER2005354}. The significance of our algorithm is highlighted by the widespread use of finite differences for spatial discretization in most MHD and two--fluid codes, largely because of their implementation simplicity.

We note that in this work, we adopt the skew--symmetric approach \cite{morinishi2010skew} to reformulate the parallel gradient operator $\ngradpar$, characterized by a divergence--free advective field, as a weighted average of the advective $\mathbf{b}\cdot \boldsymbol{\nabla}\bullet$ and divergence forms $\nabla \cdot( \mathbf{b}\bullet)$. Additionally, we establish strict relationships that connect the discretization of the operators on the two staggered grids. The concept of developing a conservative scheme of arbitrary order on staggered grids by averaging the advective and divergence forms of the convective term, thereby resulting in a skew--symmetric operator, originates from the work of \cite{morinishi1998fully}, \cite{morinishi2010skew}. Their research focused on ensuring the conservation of mass, momentum, and kinetic energy in the direct numerical simulation (DNS) of the Navier--Stokes equations.
Furthermore, it has been extended to fluid plasma models, where it was used to reformulate the hyperbolic components of the equations at the continuous level \cite{halpern2018anti}, \cite{halpern2020anti}.
More recently, Halpern et al. applied this methodology to discretize the diffusive term, enhancing spectral fidelity \cite{halpern2024paralleldiffusionoperatormagnetized}.

As a test of the numerical scheme we propose, we consider the electrostatic Shear Alfvén waves (SAWs) \cite{Jolliet:207684}, which are stable plasma waves described by a hyperbolic system of equations for the electron parallel velocity and the electrostatic potential. 
In most fluid descriptions of plasma, the SAWs constitute the fastest oscillations in the direction of the equilibrium magnetic field. We demonstrate that our scheme guarantees that the SAWs are also stable at the discrete level. We analyze the system of SAWs with the inclusion of parallel diffusion in the equation for electron parallel velocities, as this term is physically present in the two--fluid plasma model due to the gyroviscous effects. Accounting for this diffusion is essential because it influences the behavior of the SAWs.

This paper is organized as follows. After the Introduction, Sec.~\ref{sec:numerical_scheme} defines the employed operators and presents their discretization on staggered grids. Also, we construct the MFD scheme for the parallel gradient operator and we prove that our discretization preserves the divergence theorem at the discrete level. In Sec.~\ref{sec:results} these schemes are applied to solve a wave model problem in a three--dimensional setting. Additionally, we demonstrate that preserving the divergence theorem at the discrete level in the context of wave problems is necessary to achieve energy conservation in the system. Sec.~\ref{sec:saws} focuses on
applying the discretization scheme proposed in  Sec.\ref{sec:numerical_scheme} to the SAWs to assess the energy conservation of the new staggered grid operators.
The conclusions follow in Sec.~\ref{sec:concl}.
\section{Mimetic finite difference discretization of the parallel gradient on staggered grids}
\label{sec:numerical_scheme}
In this section, we construct an MFD scheme on staggered grids to discretize transport operators with divergence--free advective field $\mathbf{b}$. Namely, we define the transport operator $\ngradpar:\mathbb{R}\to \mathbb{R}$ such that $\ngradpar f=\textbf{b}\cdot \boldsymbol{\nabla}f$, with $\nabla\cdot \textbf{b}=0$; in the following, we will refer to this operator as the parallel gradient operator. At the continuous level, by taking $f=pq$, the divergence theorem states that
\begin{equation}
    \int_{\Omega} p\ngradpar q  \ d\Omega+\int_{\Omega}q\ngradpar p \  d\Omega= \int_{\partial\Omega} pq \textbf{b}\cdot \textbf{n} \ ds,
    \label{eqn:theo_div}
\end{equation}
where $p$ and $q$ are two scalar fields and $\Omega$ is a three--dimensional bounded domain with boundary $\partial \Omega$.
The proposed algorithm preserves the divergence theorem Eq.~\eqref{eqn:theo_div} at the discrete level when using staggered grids in a three--dimensional Cartesian geometry.

In wave problems a staggering between the grids on which the two different fields $p$ and $q$ of Eq.~\eqref{eqn:theo_div} are evaluated is necessary to avoid the emergence of checkerboard patterns. 
\begin{figure}[t]
    \centering
    \includegraphics[scale=0.5]{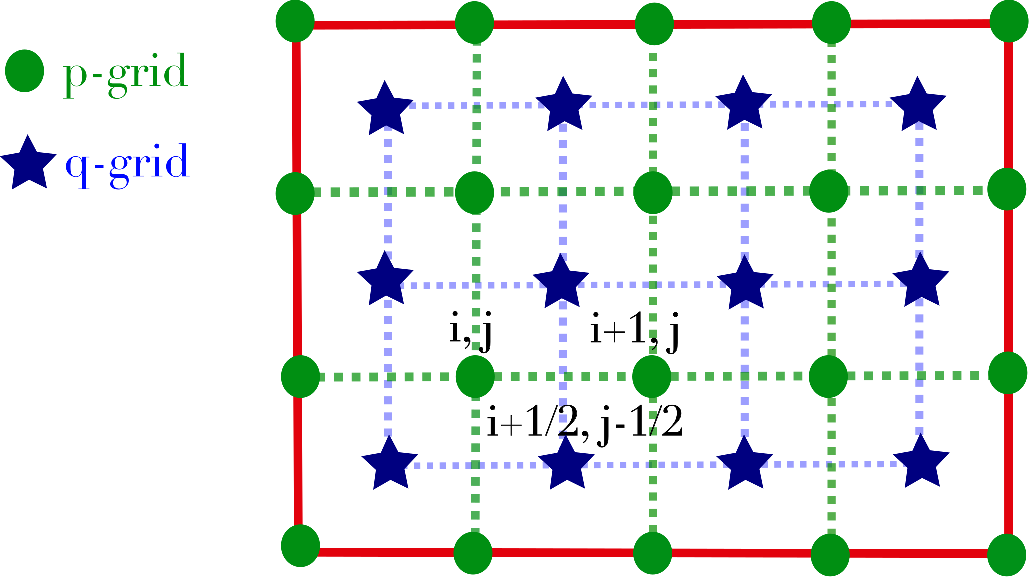}
    \caption{Sketch of the two grids in a two dimensional setting where the red line represents the physical domain, the green dashed line the $p$-grid with the green points as nodes and the blue dashed line describes the $q$-grid with blue stars as nodes.}
    \label{fig:sketch_nv}
\end{figure}
Hence, a three--dimensional domain can be discretized with two uniform Cartesian grids; one denoted as $p$-grid, whose last nodes coincide with the physical boundary of the domain, and another grid, denoted as $q$-grid, which is staggered in every direction of half of a cell, as shown in Fig.~\ref{fig:sketch_nv}. The scalar fields $p$ and $q$ are evaluated on the $p$-grid and $q$-grid respectively.
We define the set of indices in the three directions as $\mathcal{I}^\gamma=\mathcal{I}_x^\gamma \times\mathcal{I}_y^\gamma \times \mathcal{I}_z^\gamma $ with $\mathcal{I}_x^\gamma=\{1,..., N_x^\gamma\}$, $\mathcal{I}_y^\gamma=\{1,..., N_y^\gamma\}$ and $\mathcal{I}_z^\gamma=\{1,..., N_z^\gamma\}$ where $\gamma \in \{p,q\}$ and $N_x^\gamma$, $N_y^\gamma$, $N_z^\gamma$ are the nodes respectively in the $x$, $y$ and $z$ direction in the $\gamma$-grid.

Because we have variables defined on two different grids, it becomes necessary to define the parallel gradient operators that map between these grids.
Note that $\nabla\cdot \left(\textbf{b}f\right)$ at the continuous level is equal to $\textbf{b}\cdot \boldsymbol{\nabla }f$, since $\nabla\cdot \textbf{b}=0$. Starting from the definition of the parallel gradient, we define the discrete operators with the help of the staggered indices $i^+=i+\frac{1}{2}$, $j^+=j+\frac{1}{2}$, $k^+=k+\frac{1}{2}$ and $i^-=i-\frac{1}{2}$, $j^-=j-\frac{1}{2}$, $k^-=k-\frac{1}{2}$ and $\Vec{f}^p$, $\Vec{f}^q$ as the vectors associated with the scalar field $f$
\begin{equation}
    \begin{alignedat}{2}
    &\Vec{f}^p\in \mathbb{R}^{N_x^p N_y^p N_z^p}: \left(\Vec{f}^p\right)_{\mathcal{F}^p(i,j,k)}=\left(f(x_i,y_j,z_k)\right)_{i,j,k} \text{ on the $p$-grid}, \\
    & \Vec{f}^q \in \mathbb{R}^{N_x^q N_y^q N_z^q}:  \left(\Vec{f}^q\right)_{\mathcal{F}^q(i^+,j^+,k^+)}=\left(f(\xijkv)\right)_{i^+,j^+,k^+} \text{ on the $q$-grid},
 \label{eqn:vecform}
 \end{alignedat}
 \end{equation}
where $\mathcal{F}^\gamma:\mathbb{R}^{N_x^\gamma \times N_y^\gamma \times N_z^\gamma}\to \mathbb{R}^{N_x^\gamma N_y^\gamma N_z^\gamma}$ is a suitable indexing.
Considering that \begin{equation*}
    \mathbf{b}= \left[b^x(x,y,z), b^y(x,y,z),b^z(x,y,z)\right]^T
\end{equation*}
is a given advective field that can be evaluated on the $p$-grid and on the $q$-grid, we define the operator $\textbf{b} \cdot \boldsymbol{\nabla}\vert_{pq}f:q$-grid $\rightarrow p$-grid as:
\begin{equation}
\begin{alignedat}{2}
    \left[\textbf{b} \cdot \boldsymbol{\nabla}\vert_{pq}  f\right](x_i,y_j,z_k)=&b^x (x_i,y_j,z_k) \left[D_x \vert_{pq}f\right]\left(x_i,y_j,z_k\right) +\\&b^y(x_i,y_j,z_k) \left[D_y \vert_{pq}f\right]\left(x_i,y_j,z_k\right)+\\& b^z(x_i,y_j,z_k) \left[ D_z \vert_{pq}f\right]\left(x_i,y_j,z_k\right), 
    \label{eqn:parallelgradq2p}
\end{alignedat}
\end{equation}
where the operators $D_x$, $D_y$, and $D_z$ approximate the individual derivatives using a centered finite difference scheme. In this case the advective field $\textbf{b}$ is evaluated on the $p$-grid and the operators $D_\delta |_{pq} f: q$-grid $\rightarrow p$-grid with $\delta \in \{x,y,z\} $ take values of $f$ on the $q$-grid and produce results on the $p$-grid.
 We define the operator $\nabla \vert_{pq} \cdot\left(\textbf{b}f\right):q$-grid $\rightarrow p$-grid applied to $f$ on the $q$-grid and evaluated in $(x_i,y_j,z_k)$ as
\begin{equation}
\begin{alignedat}{2}
   \left[ \nabla \vert_{pq} \cdot \left(\mathbf{b} f\right)\right](x_i,y_j,z_k)=&\left[D_x\vert_{pq}\left(b^x f\right)\right](x_i,y_j,z_k)+\\&\left[D_y\vert_{pq}\left(b^y f\right)\right](x_i,y_j,z_k)+\\&\left[D_z\vert_{pq} \left(b^z f\right)\right](x_i,y_j,z_k).
    \label{eqn:div_q2p}
\end{alignedat}
\end{equation}
where the operators $D_\delta |_{pq}(b^\delta f):q$-grid $\rightarrow p$-grid take values of $b^\delta f$ in the $q$-grid and produce results on the $p$-grid.
We also define the operator $\textbf{b}\cdot \boldsymbol{\nabla}\vert_{qp}f: p$-grid $\rightarrow q$-grid as
\begin{equation}
\begin{alignedat}{2}
    \left[\mathbf{b}\cdot \boldsymbol{\nabla}\vert_{qp} f\right] (\xijkv)=&b^x (\xijkv) \left[D_x\vert_{qp} f\right]\left(\xijkv\right) +\\&b^y(\xijkv)\left[D_y \vert_{qp}f\right]\left(\xijkv\right)+\\& b^z (\xijkv)\left[D_z \vert_{qp} f\right]\left(\xijkv\right),
    \label{eqn:parallelgradp2q}
\end{alignedat}
\end{equation}
and the operator $\nabla \vert_{qp} \cdot( \mathbf{b} f):p$-grid $\rightarrow q$-grid applied to $f$ living in the $p$-grid and evaluated in $(\xijkv)$ as
\begin{equation}
\begin{alignedat}{2}
     \left[\nabla \vert_{qp} \cdot \left(\textbf{b} f\right)\right](\xijkv)=&\left[D_x\vert_{qp}\left(b^x f\right)\right](\xijkv) +\\&\left[D_y\vert_{qp}\left(b^y f\right)\right](\xijkv)+\\&\left[D_z\vert_{qp}\left(b^z  f\right)\right](\xijkv).
     \label{eqn:div_p2q}
\end{alignedat}
\end{equation}
With the help of these discrete operators, for $\alpha$, $\beta \in \left[0,1\right]$, we define the parallel gradient $ \ngradpar\vert_{pq}  f: q$-grid $\rightarrow$ $p$-grid as a weighted average of two operators $(\textbf{b}\cdot \boldsymbol{\nabla}\vert_{pq}) \bullet$ and $\nabla \vert_{pq} \cdot (\textbf{b} \bullet)$:
\begin{equation}
\begin{alignedat}{2}
    \ngradpar\vert_{pq} f (x_i,y_j,z_k)=&\alpha\left[\textbf{b} \cdot \boldsymbol{\nabla}\vert_{pq} f\right] (x_i,y_j,z_k) +\\& (1-\alpha)\left[\nabla \vert_{pq} \cdot \left(\textbf{b} f\right)\right](x_i,y_j,z_k).
    \label{eqn:weigter_ngradpar_pq}
 \end{alignedat}
\end{equation}
It is possible to rewrite the operator $\ngradpar\vert_{pq}$ applied to $f$ in a matrix--vector form as $\textbf{C}_{pq} \Vec{f}^q$ where $\Vec{f}^q$ is the vector associated with the scalar field $f$ as in Eq.~\eqref{eqn:vecform} and, the matrix associated with the operator $\ngradpar\vert_{pq}$ is $\textbf{C}_{pq}=\left(\alpha\textbf{A}_{pq}+(1-\alpha)\textbf{B}_{pq}\right)$ where the matrices $\textbf{A}_{pq}$ and $\textbf{B}_{pq}$ are associated, respectively, with the discretized operators $(\textbf{b} \cdot \boldsymbol{\nabla}\vert_{pq})\bullet$ and $\nabla\vert_{pq}\cdot \left(\textbf{b} \bullet\right)$.
The parallel gradient $\ngradpar\vert_{qp} f: p$-grid $\rightarrow$ $q$-grid is defined as a weighted average of two operators $(\textbf{b} \cdot \boldsymbol{\nabla}\vert_{qp})\bullet$ and $\nabla \vert_{qp}\cdot (\textbf{b} \bullet)$: 
\begin{equation}
\begin{alignedat}{2}
    \ngradpar\vert_{qp} f (\xijkv)=&\beta \left[\mathbf{b} \cdot \boldsymbol{\nabla}\vert_{qp} f\right] (\xijkv)+ \\& (1-\beta)\left[\nabla \vert_{qp} \cdot \left(\mathbf{b} f\right)\right](\xijkv) .
     \label{eqn:weigter_ngradpar_qp}
 \end{alignedat}
\end{equation}
Similarly, it is possible to rewrite the operator $\ngradpar\vert_{qp}$ applied to $f$ in a matrix--vector form as $\textbf{C}_{qp}\Vec{f}^p$ where $\Vec{f}^p$ is the vector associated with the scalar field $f$ as in Eq.~\eqref{eqn:vecform} and, the matrix associated with the operator $\ngradpar\vert_{qp}$ is $\textbf{C}_{qp}=\left(\alpha\textbf{A}_{qp}+(1-\alpha)\textbf{B}_{qp}\right)$ where the matrices $\textbf{A}_{qp}$ and $\textbf{B}_{qp}$ are associated, respectively, with the discretized operators $(\textbf{b}  \cdot \boldsymbol{\nabla}\vert_{qp})\bullet$ and $\nabla\vert_{qp}\cdot \left( \textbf{b}\bullet\right)$.

We note that the overall convergence order of the discretization scheme depends on the discretization of the derivatives along the $x$, $y$, and $z$ directions.
In this context, we provide the implementation of the $x$-direction derivative, $D_x$, which is second--order accurate. Of course, higher--order differences can also be employed and the implementations of the derivatives in the other directions are analogous.
The operator $D_x\vert_{pq}(\bullet)$ applied to a function $f$ evaluated in $(x_i,y_j,z_k)$ is defined as:
\begin{equation}
    \begin{alignedat}{2}
      \left[D_x\vert_{pq} f\right](x_i,y_j,z_k)= &\frac{1}{4\Delta x}\bigg[\bigg(f(x_{i^+},y_{j^-},z_{k^-})+f(x_{i^+},y_{j^+},z_{k^+})+\\&f(x_{i^+},y_{j^+},z_{k^-})+f(x_{i^+},y_{j^-},z_{k^+})\bigg)-\\&\bigg(f(x_{i^-},y_{j^-},z_{k^-})+f(x_{i^-},y_{j^+},z_{k^+})+\\&f(x_{i^-},y_{j^+},z_{k^-})+f(x_{i^-},y_{j^-},z_{k^+})\bigg)\bigg],
      \label{eqn:dx_q2p}
    \end{alignedat}
\end{equation}
while the operator $D_x\vert_{qp}(\bullet)$ applied to $f$ and evaluated in $(\xijkv)$ is defined in the following way:
\begin{equation}
\begin{alignedat}{2}
   \left[D_x\vert_{qp} f\right](\xijkv)= &\frac{1}{4\Delta x}\bigg[\bigg(g(x_{i+1},y_j,z_k)+g(x_{i+1},y_{j+1},z_{k+1})+\\&g(x_{i+1},y_{j+1},z_k)+g(x_{i+1},y_j,z_{k+1})\bigg)-\\&\bigg(g(x_{i},y_j,z_k)+g(x_{i},y_{j+1},z_{k+1})+\\&g(x_{i},y_{j+1},z_k)+g(x_{i},y_j,z_{k+1})\bigg)\bigg].
   \label{eqn:dx_p2q}
\end{alignedat}
\end{equation} 
\begin{rmk}
The accuracy of the operators $\nabla_\parallel \vert (\bullet)$ is guaranteed by the fact that the operators are defined as a weighted average of two operators that are both second--order accurate in space. This makes the proposed discretization of the parallel gradient second--order accurate in space.
\end{rmk}

We now prove that the choice of $\beta=1-\alpha$ ensures that Eq.~\eqref{eqn:theo_div} is verified at the discrete level in the case of homogeneous Dirichlet boundary conditions. 
Considering $p$ and $q$ scalar fields evaluated respectively on the $p$ and $q$-grid, we introduce the matrices $\textbf{X}_p\in \mathbb{R}^{N_x^pN_y^pN_z^p \times N_x^pN_y^pN_z^p}$ and $\textbf{X}_q\in \mathbb{R}^{N_x^qN_y^qN_z^q\times N_x^qN_y^qN_z^q}$ that are diagonal and positive definite and allows to compute the discrete $L^2$ norms in the two grids:
\begin{equation}
\begin{alignedat}{2}
    &\|\Vec{p}\|^2=\Vec{p}^T \textbf{X}_p \Vec{p}=\Delta x\Delta y \Delta z \sum_{i\in \mathcal{I}_x^p}\sum_{j\in \mathcal{I}_y^p}\sum_{k\in \mathcal{I}_z^p}w_{i,j,k}^p(p_{i,j,k})^2\approx \int_{\Omega} p^2 \ d\Omega,\\
    &\|\Vec{q}\|^2=\Vec{q}^T \textbf{X}_q \Vec{q}=\Delta x\Delta y \Delta z \sum_{i\in \mathcal{I}_x^q}\sum_{j\in \mathcal{I}_y^q}\sum_{k\in \mathcal{I}_z^q}w_{i^+,j^+,k^+}^q(q_{i^+,j^+,k^+})^2\approx \int_{\Omega} q^2 \ d\Omega,
    \label{eqn:norm}
\end{alignedat}
\end{equation}
where $w_p$ and $w_q$ are the weights for the quadrature formula. 
The weights in the interior points are usually equal to 1, as shown in \cite{o2017energy}.
In the general three--dimensional case, assuming the use of the trapezoidal rule for numerical integration, the matrix $\textbf{X}_p$ is a diagonal matrix whose diagonal entries are the weights $w_{i,j,k}^p$. These weights are determined by the location of the points within the domain: they are $(\Delta x \Delta y \Delta z)/8$ at the corners, $(\Delta x \Delta y \Delta z )/4$ along the edges,  $(\Delta x \Delta y \Delta z)/2$ on the faces, and $(\Delta x \Delta y \Delta z)$ at the interior points.
On the other hand, as shown in Fig.~\ref{fig:sketch_nv}, since there are no degrees of freedom on the boundary for the $q$-grid, the matrix $\textbf{X}_q$ is the identity matrix scaled by $(\Delta x\Delta y\Delta z)$.

We note that the boundary conditions can modify the weights of $\textbf{X}_p$. For instance, if periodic boundary conditions are applied in one or more directions, the treatment of boundary points changes accordingly, which may lead to adjustments in the corresponding weights.
\begin{thm}
If $\alpha+\beta=1$, then the divergence theorem Eq.~\eqref{eqn:theo_div} with homogeneous Dirichlet boundary conditions is preserved in the discrete setting. Moreover, $\textbf{C}_{pq}=-\textbf{C}_{qp}^T$.
    \label{theo:energy}
\end{thm}
\begin{proof}
Because of the homogeneous boundary conditions, the right--hand side of Eq.~\eqref{eqn:theo_div} vanishes. 
Consequently, at the discrete level, the corresponding discretized energy $E_d$ is expected to satisfy: 
\begin{equation}
E_d=\Vec{p}^T\textbf{X}_p\textbf{C}_{pq} \Vec{q}+ \Vec{q}^T\textbf{X}_q\textbf{C}_{qp}\Vec{p}=0,
\label{eqn:proof_den}
\end{equation} 
where $\Vec{p}$ and $\Vec{q}$ are the vectors of size respectively $N_x^p\cdot N_y^p \cdot N_z^p $ and $N_x^q\cdot N_y^q \cdot N_z^q $ containing all the degrees of freedom associated with $p(x_i,y_j,z_k)$ and $q(x_{i^+},y_{j^+},z_{k^+})$ as in Eq.~\eqref{eqn:vecform}. 
In the following, we derive a condition on $\alpha$ and $\beta$ under which Eq.~\eqref{eqn:proof_den} holds.
This equation can be rewritten component--wise for the interior point as:
\begin{equation}
\displaystyle\sum_{(i,j,k) \ \in \mathcal{I}^p} p_{i,j,k}(\ngradpar \vert_{pq} q)_{i,j,k}  +
    \displaystyle\sum_{(i^+,j^+,k^+) \ \in \ \mathcal{I}^q} q_{\ijkv}  \left(\ngradpar\vert_{qp} p\right)_{\ijkv}=0,
    \label{eqn:rhs_energy}
\end{equation}
since the weights $w^p$ and $w^q$ in Eq.~\eqref{eqn:norm} are 1 for the interior points.

Without loss of generality and considering that $E_d=E_d^x+E_d^y+E_d^z$, we can focus our analysis on the contribution to the energy from the first component, $x$, of the gradient in Eq.~\eqref{eqn:rhs_energy}, that is:
\begin{equation}
\begin{alignedat}{2}
    E_d^x=&\displaystyle\sum_{(i,j,k) \ \in \mathcal{I}^p} p_{i,j,k}(\ngradpar \vert_{pq}^x q)_{i,j,k}  +
    \displaystyle\sum_{(i^+,j^+,k^+) \ \in \ \mathcal{I}^q} q_{\ijkv}  \left(\ngradpar\vert_{qp}^x p\right)_{\ijkv}\\
    =&\displaystyle\sum_{(i,j,k) \ \in \mathcal{I}^p} \underbrace{p_{i,j,k}\alpha \left(b^x_{i,j,k} D_x\vert_{pq}q (x_i,y_j,z_k)\right)}_{L}+\\& \displaystyle\sum_{(i,j,k) \ \in \mathcal{I}^p} \underbrace{p_{i,j,k} (1-\alpha) \left(D_x\vert_{pq}\left(b^x  q\right)(x_i,y_j,z_k)\right)}_{M}+\\ &\displaystyle\sum_{(i^+,j^+,k^+) \ \in \mathcal{I}^q}\underbrace{q_{\ijkv}\beta \left(b^x_{\ijkv} D_x\vert_{qp}p(\xijkv)\right)}_{N}+\\&\displaystyle\sum_{(i^+,j^+,k^+) \ \in \mathcal{I}^q} \underbrace{q_{\ijkv}(1-\beta)\left(D_x\vert_{qp}\left(b^x p\right)(\xijkv)\right)}_{O}.
    \label{eqn:discret_x}
\end{alignedat}
\end{equation}

We can further develop Eq.~\eqref{eqn:discret_x} inserting the definition of the operators $D_x\vert_{pq}\bullet$ and $D_x\vert_{qp}\bullet$ reported in Eq.~\eqref{eqn:dx_q2p} and Eq.~\eqref{eqn:dx_p2q}  in the following way:
\begin{equation*}
\begin{alignedat}{2}
&E_d^x=\displaystyle\sum_{(i,j,k) \ \in \mathcal{I}^p} \bigg[\frac{p_{i,j,k}}{4\Delta x} \alpha b^x_{i,j,k}\bigg[\bigg(\overbrace{q_{i^+,j^-,k^-}}^{1}+q_{i^+,j^+,k^+} +q_{i^+,j^+,k^-}+q_{i^+,j^-,k^+}\bigg)\\ & -\bigg(q_{i^-,j^-,k^-}+q_{i^-,j^+,k^+} +q_{i^-,j^+,k^-}
+q_{i^-,j^-,k^+}\bigg)\bigg]\\& +\frac{p_{i,j,k}}{4\Delta x} (1-\alpha)\bigg[\bigg(\overbrace{(b^x  q)_{i^+,j^-,k^-}}^{2}+(b^x q)_{i^+,j^+,k^+}+(b^x q)_{i^+,j^+,k^-}+(b^x  q)_{i^+,j^-,k^+}\bigg)\\&   -\bigg((b^x q)_{i^-,j^-,k^-}+(b^x q)_{i^-,j^+,k^+}+(b^x q)_{i^-,j^+,k^-})+(b^x q)_{i^-,j^-,k^+}\bigg)\bigg]\bigg]
\\
  & +\displaystyle\sum_{(i^+,j^+,k^+) \ \in \mathcal{I}^q}\bigg[\frac{q_{i^+,j^+,k^+}}{4\Delta x}\beta b^x_{i^+,j^+,k^+}\bigg[\bigg(p_{i+1,j,k}+p_{i+1,j+1,k}+p_{i+1,j,k+1}\\&+p_{i+1,j+1,k+1}\bigg)   -\bigg(\overbrace{p_{i,j,k}}^{2}+p_{i,j+1,k}+p_{i,j,k+1}+p_{i,j+1,k+1}\bigg)\bigg]\\ & 
   +\frac{q_{i^+,j^+,k^+}}{4\Delta x} (1-\beta)\bigg[\bigg((b^x p)_{i+1,j,k}+(b^x p)_{i+1,j+1,k}+(b^x p)_{i+1,j,k+1}+\\&(b^x p)_{i+1,j+1,k+1}\bigg)  -\bigg(\overbrace{(b^x p)_{i,j,k}}^{1}+(b^x p)_{i,j+1,k}+(b^x p)^{i,j,k+1}+(b^x p)_{i,j+1,k+1}\bigg)\bigg]\bigg].
\end{alignedat}
\end{equation*}

By writing out all the contributions, we observe that the terms arising from the sum denoted by $L$ in Eq.~\eqref{eqn:discret_x} appear in the sum denoted by $O$, with opposite signs and are scaled by $(1-\beta)$ instead of $\alpha$. Moreover, we find the same contribution due to the sum $M$ in the sum $N$ with opposite sign and multiplied by $\beta$ instead of $(1-\alpha)$. 
\end{proof}
\begin{rmk}
If we consider the discretization formula of the parallel Laplacian $\nlaplpar$ as reported in \cite{gunter2007finite} and we choose $\alpha=0$ and $\beta=1$, the composition of the two operators $\ngradpar\vert_{qp}$ and $\ngradpar\vert_{pq}$ and so the resulting matrix $\textbf{C}_{qp}\textbf{C}_{pq}$ coincides with the matrix associated with the operator $\nlaplpar$ on the $q$-grid. Similarly, if $\alpha=1$ and $\beta=0$ the composition of $\ngradpar\vert_{pq}$ and $\ngradpar\vert_{qp}$ is equal to the matrix describing the $\nlaplpar$ on the $p$-grid.
\end{rmk}
\section{Wave model problem}
\label{sec:results}
\hspace*{0.3cm}
As the first model problem, we consider a wave that propagates parallel to $\textbf{b}$:
\begin{equation}
\begin{alignedat}{2}
\begin{cases}
\pdv{p}{t}=\ngradpar q,&\\
\pdv{q}{t}=\ngradpar p.
\end{cases}
 \label{eqn:model}
\end{alignedat}
\end{equation}
on $\Omega \subset \mathbb{R}^3$ which is equivalent to 
\begin{equation}
    \frac{\partial^2 p}{\partial t}= \frac{\partial }{\partial t}\left(\textbf{b}\cdot \boldsymbol{\nabla} q\right)=\textbf{b}\cdot \boldsymbol{\nabla} \left(\pdv{q}{t}\right)=\textbf{b}\cdot \boldsymbol{\nabla} \left(\textbf{b}\cdot \boldsymbol{\nabla} p \right)= \nlaplpar p.
    \label{eqn:wave}%
\end{equation}
where $\nlaplpar$ is the parallel Laplacian defined as $\nlaplpar: \mathbb{R}\to \mathbb{R}$ such that $\nlaplpar f =\mathbf{b}\cdot \boldsymbol{\nabla } \left(\mathbf{b}\cdot \boldsymbol{\nabla }f \right)$.
The system describes a wave propagating in the direction of $\textbf{b}$. 
 
 \subsection{Boundary conditions}
\label{sec:bc}
We discuss the boundary conditions to impose to Eqs.~\eqref{eqn:model}. For this purpose, we apply the change of variables $r_1=p+q$ and $r_2=p-q$.
The model problem in Eqs.~\eqref{eqn:model} can then be rewritten as:
\begin{equation}
    \begin{alignedat}{2}
    \begin{cases}
        \pdv{r_1}{t}&=\textbf{b}\cdot\boldsymbol{\nabla}r_1,\\
        \pdv{r_2}{t}&=-\textbf{b}\cdot\boldsymbol{\nabla}r_2,
   \end{cases}
    \end{alignedat}
    \label{eqn:model_decoup}
\end{equation}
where the variables $r_1$ and $r_2$ are decoupled. Eq.~\eqref{eqn:model_decoup} requires that the boundary conditions are imposed for $r_1$ on the inlet part of the domain $\Gamma_{-}=\{x\in\partial\Omega : \textbf{b}\cdot \textbf{n}<0\}$, and for $r_2$ on the outlet part $\Gamma_{+}=\{x\in\partial\Omega : \textbf{b}\cdot \textbf{n}>0\}$, where $\textbf{n}$ is the normal vector pointing out of the domain.
\begin{figure}[t]
    \centering
    \includegraphics[scale=0.65]{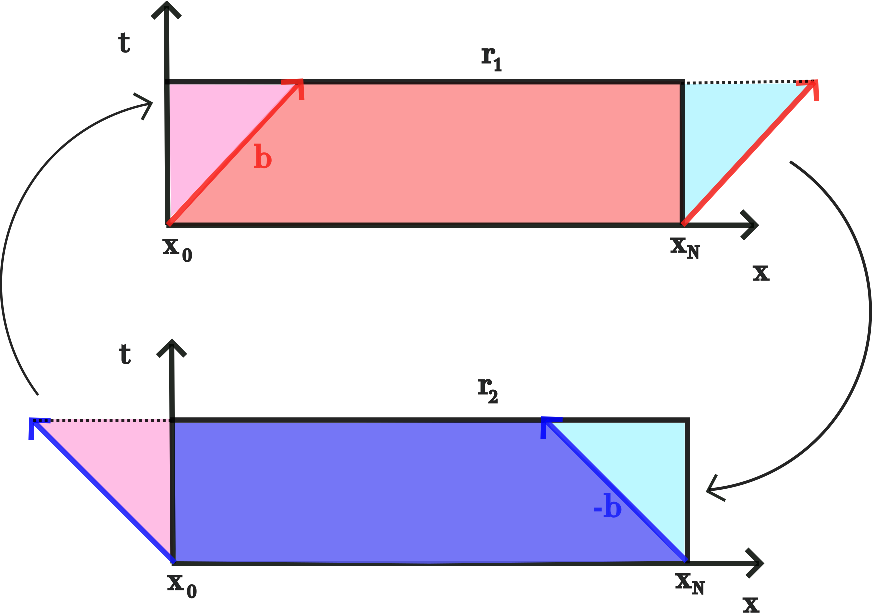}
    \caption{Schematic 1D solution of the problem Eq.~\eqref{eqn:model_decoup} in a $x$-$t$ space with the application of the boundary conditions. In this case, the inflow of $r_1$ is $x_0$ and the outflow is $x_{N}$. For $r_2$ the inflow is $x_{N}$ where the information arrives from the outflow of $r_1$ and the outflow is $x_0$.}
    \label{fig:bc}
\end{figure}
A possible choice of the boundary conditions is $r_1= -r_2$ on $\Gamma_-$ and $r_2= -r_1$ on $\Gamma_+ $
that is $r_1+r_2=0$ on $ \partial \Omega$. 
One possibility to satisfy $r_1+r_2=0$ is to impose $p=0$ on $\partial \Omega$, while $q$ is left free on the boundary.

A schematic version of the process in a one--dimensional setting is reported in Fig.~\ref{fig:bc}. 
For our specific test case, we consider a three--dimensional domain, where we impose $p=0$, on the boundaries of the $x$-$y$ planes for every $z$. In the $z$ direction, we apply periodic boundary conditions to simulate the tokamak domain in our application.

\subsection{Energy conservation}
\label{sec:ener_model}
We now turn to the energy conservation properties of the model in Eqs.~\eqref{eqn:model}, an important feature that the discretized system has to retain.
Starting from Eqs.~\eqref{eqn:model}, we multiply the first equation by $p$ and the second by $q$, and we integrate them in space over a domain $\Omega$. Summing up, we obtain
\begin{equation}
  \int_{\Omega}p\pdv{p}{t} +q\pdv{q}{t} \ d\Omega=\int_{\Omega} \left(p\textbf{b}\cdot \boldsymbol{\nabla} q+q\textbf{b}\cdot \boldsymbol{\nabla} p \right)\ d\Omega,
  \label{eqn:energy1}
\end{equation}
which is equivalent to
\begin{equation}
 \frac{\partial}{\partial t}\int_{\Omega} \left(\frac{p^2}{2}+\frac{q^2}{2}\right)\ d\Omega=\int_{\partial\Omega } pq\textbf{b}\cdot \textbf{n}\ d s=0.
    \label{eqn:conservation_model}
\end{equation}
since $p=0$ on $\partial \Omega$. Eq.~\eqref{eqn:conservation_model} express the conservation in time of the quantity $E=\int_{\Omega}\left(p^2/2+q^2/2\right)\ d\Omega$ that we define as energy.

\subsection{Energy conservation of the space-discretized system}
Problem \eqref{eqn:model} is linear and it can be discretized in space as outlined in Sec.~\ref{sec:numerical_scheme}.
Eqs.~\eqref{eqn:model} can be written in a matrix form as 
\begin{equation}
    \frac{d}{dt}
\left[\begin{array}{c}
\Vec{p} \\
\Vec{q}
\end{array}\right]
=\left[\begin{array}{c|c}
\textbf{0} & \textbf{C}_{pq} \\
\hline
\textbf{C}_{qp} & \textbf{0}
\end{array}
\right]\left[\begin{array}{c}
\Vec{p} \\
\Vec{q}
\end{array}\right].
\label{eqn:discretized_eq}
\end{equation}
\begin{corollary}
If $\alpha+\beta=1$, then the discrete problem \eqref{eqn:discretized_eq} is energy--preserving, i.e. $\frac{d}{dt}\int_\Omega E \ d\Omega=0$.
    \label{cor:energy}
\end{corollary}
\begin{proof}
We notice that
\begin{equation*}
    \frac{1}{2}\frac{d\|\Vec{x}\|^2}{dt}=\frac{1}{2}\frac{d}{dt}(\|\Vec{p}\|^2+\|\Vec{q}\|^2)=\Vec{p}^T\textbf{X}_p\textbf{C}_{pq}\Vec{q}+\Vec{q}^T\textbf{X}_q\textbf{C}_{qp}\Vec{p}.
\end{equation*}
and that
\begin{equation}
\Vec{p}^T\textbf{X}_p\textbf{C}_{pq} \Vec{q}+ \Vec{q}^T\textbf{X}_q\textbf{C}_{qp}\Vec{p}=0,
\end{equation} 
is a direct consequence of Theorem~\ref{theo:energy}.
\end{proof}
Referring to Theorem~\ref{theo:energy}, we note that if one employs a straightforward discretization of the operator $\textbf{b}\cdot \boldsymbol{\nabla}$ with $\alpha=1$ and $\beta=1$, the resulting method lacks energy conservation and exhibits instability as $t\rightarrow\infty$. It is noteworthy that transforming the original problem described by Eqs.~\eqref{eqn:model} into flux form, representing the right--hand side as the divergence of the product between $\textbf{b}$ and the scalar variable, and subsequently employing a direct discretization of the operators with $\alpha=0$ and $\beta=0$ does not yield energy conservation either.

\subsection{Numerical test}
We consider
\begin{equation}
\begin{alignedat}{2}
    \boldsymbol{b}=&\left[\partial_y \Psi, -\partial_x \Psi, -1\right]\\
    \Psi=&\frac{1}{2}A_{mag} Ei\left(-\frac{((x-x_{mag})^2+(y-y_{mag1})^2)}{a_s^{2}}\right)-\\&\frac{1}{2}A_{mag}\log\left((x-x_{mag})^2+(y-y_{mag1})^2\right)-\\&\frac{1}{2}A_{mag}\log\left((x-x_{mag})^2+(y-y_{mag2})^2\right);
    \label{eqn:psi}
\end{alignedat}
\end{equation}
where $Ei(x)=\int_{- \infty}^x \frac{e^t}{t} \ dt$ and $A_{mag}$, $x_{mag}$, $y_{mag1}$, $y_{mag2}$ and $a_s$ are parameters reported in \ref{app1}.
Fig.~\ref{fig:quiver} reports the $xy$ components of the $\textbf{b}$ vector field.
\begin{figure}[t]
    \centering
        \vspace{-1cm}
    \includegraphics[scale=0.37]{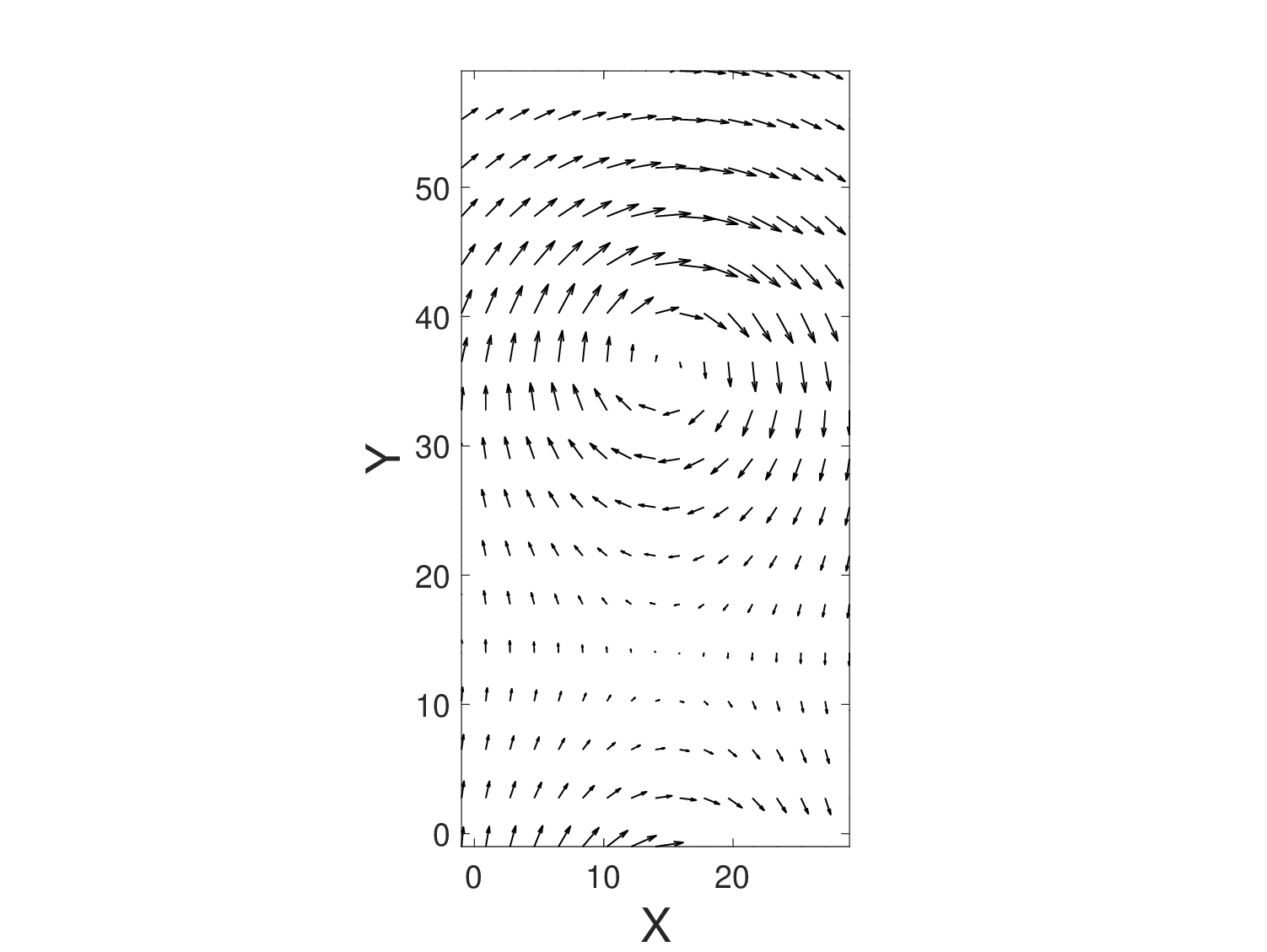}
    \caption{The quiver plot of the $x$-$y$ component of the vector field $\textbf{b}$ with a space grid $N_x \times N_y \times N_z=16\times16\times16$. The $z$ component is independent of the position and is always equal to -1.}
    \label{fig:quiver}
\end{figure}
The choice of the boundary condition in Sec.~\ref{sec:bc} is coherent with the choice of the nodes in the grids as shown in Fig.~\ref{fig:sketch_nv}. On the $q$-grid, boundary points are excluded from the degrees of freedom. On the $p$-grid, they are included to allow flexibility in imposing boundary conditions. Operators are constructed on all points, but for homogeneous Dirichlet conditions, boundary values are set to vanish, making the approach effectively equivalent to considering only the interior points as true degrees of freedom.
\begin{figure}[t]
    \centering
          \begin{subfigure}[b]{0.48\textwidth}
          \hspace{-2mm}
    \includegraphics[width=\textwidth]{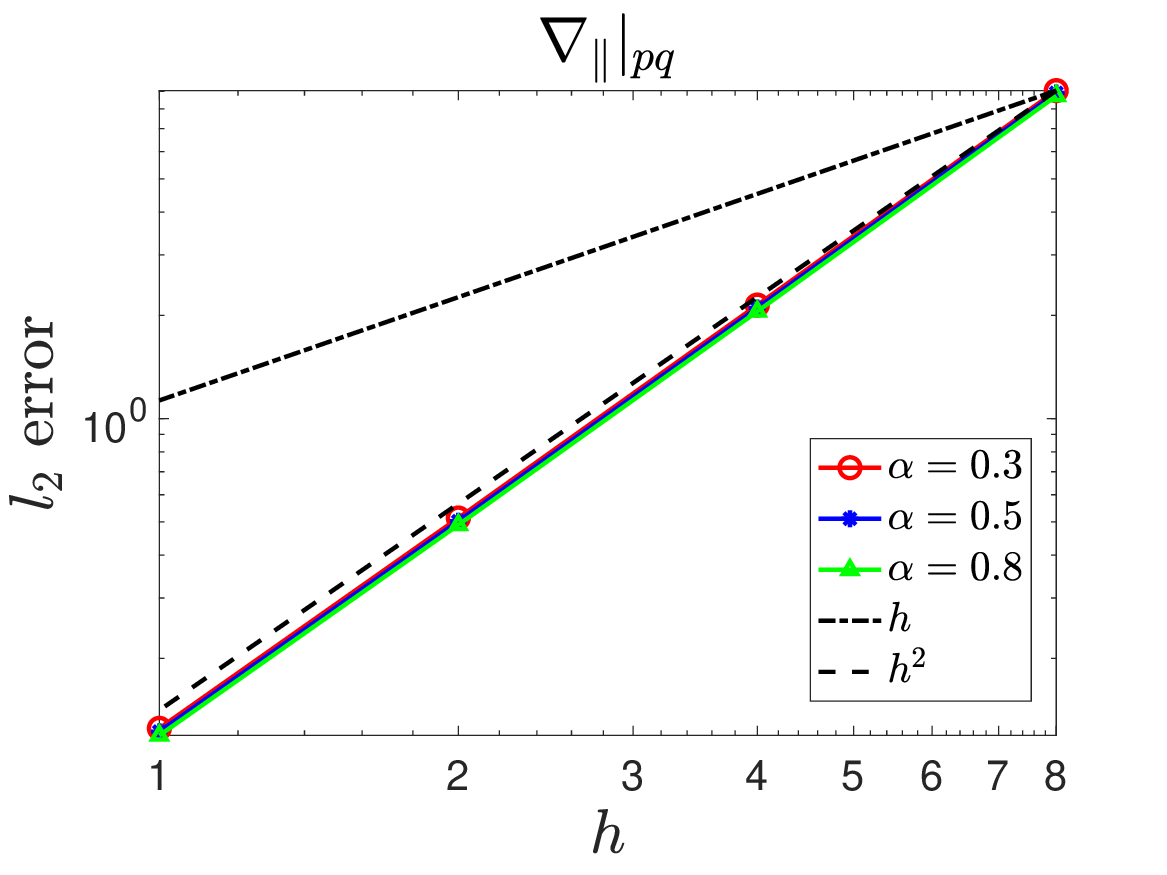}
    \caption{}
    \label{fig:convergence_gradpar_v2n}
    \end{subfigure}
    \begin{subfigure}[b]{0.48\textwidth}
    \includegraphics[width=\textwidth]{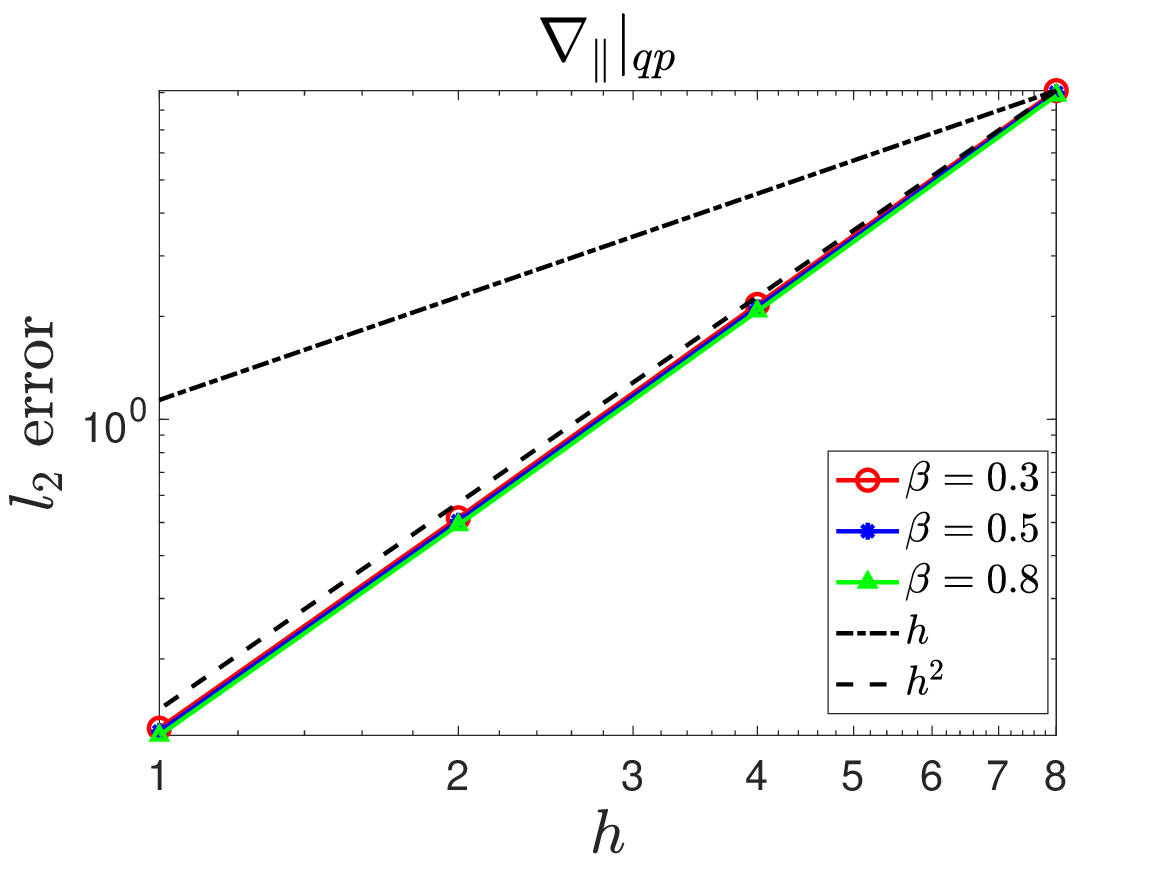}
    \caption{}
    \label{fig:convergence_gradpar_n2v}
    \end{subfigure}
    \caption{Convergence study of the operator $\ngradpar|_{pq}$ in Fig.~\ref{fig:convergence_gradpar_v2n} and operator $\ngradpar |_{qp}$ in Fig.~\ref{fig:convergence_gradpar_n2v}, with $\textbf{b} $ given in Eq.~\eqref{eqn:psi} and uniform spatial resolution such that $h = \frac{\Delta x}{\Delta x_0} \times \frac{\Delta y}{\Delta y_0}\times \frac{\Delta z}{\Delta z_0}$ where $\Delta x_0$, $\Delta y_0$ and $\Delta z_0$ are the grid spacings for $N_x\times N_y\times N_y= 64\times 64\times 64$.}
    \label{fig:convergence_gradpar}
\end{figure}

A spatial convergence study is carried out for the operators $\ngradpar|_{pq}$ and $\ngradpar|_{qp}$ across different values of $\alpha$ and $\beta$. Fig.~\ref{fig:convergence_gradpar} presents the $l_2$ norm of the error between the numerical and analytical solutions, by applying the operators $\ngradpar|_{pq}$ and $\ngradpar|_{qp}$ to the test function $f = \sin\left(2\pi x/L_x\right)\sin\left(2\pi y/L_y\right)\sin(z)$. The advective field $\textbf{b}$ is defined as in Eq.~\eqref{eqn:psi}. The results demonstrate second--order convergence in space, which aligns with the expected accuracy of the centered finite difference schemes used to approximate the derivatives in Eqs.~\eqref{eqn:dx_q2p} and \eqref{eqn:dx_p2q}.

To analyze the stability of the numerical solution, we solve the eigenvalue problem 
\begin{equation}
    \textbf{T}\Vec{x}=\lambda\Vec{x}, \quad \text{with }\textbf{T}=\left[\begin{array}{c|c}
\textbf{0} & \textbf{C}_{pq} \\
\hline
\textbf{C}_{qp} & \textbf{0}
\end{array}
\right]\ \text{and }\Vec{x}=\left[\begin{array}{c}
\Vec{p} \\
\Vec{q}
\end{array}\right],
    \label{eqn:eigen_probl}
\end{equation}
using the MATLAB command \texttt{eig}, as in \cite{du2013robust}. Our spatial grid is $N_x \times N_y \times N_z=16\times 16\times16$.
\begin{figure}[t]
    \centering
          \begin{subfigure}[b]{0.6\textwidth}
    \includegraphics[scale=0.3]{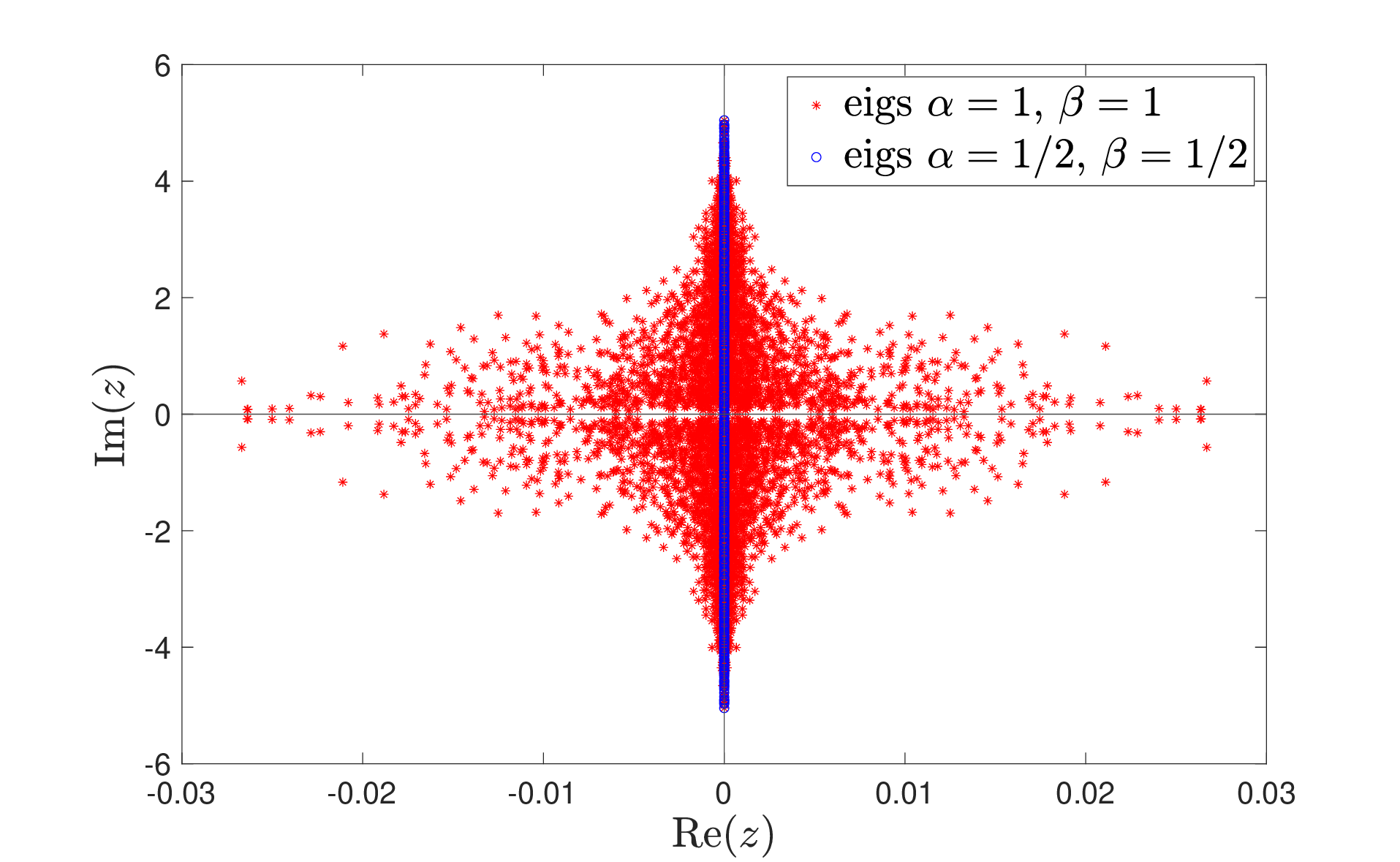}
    \end{subfigure}
    \begin{subfigure}[b]{0.35\textwidth}
    \includegraphics[scale=0.3]{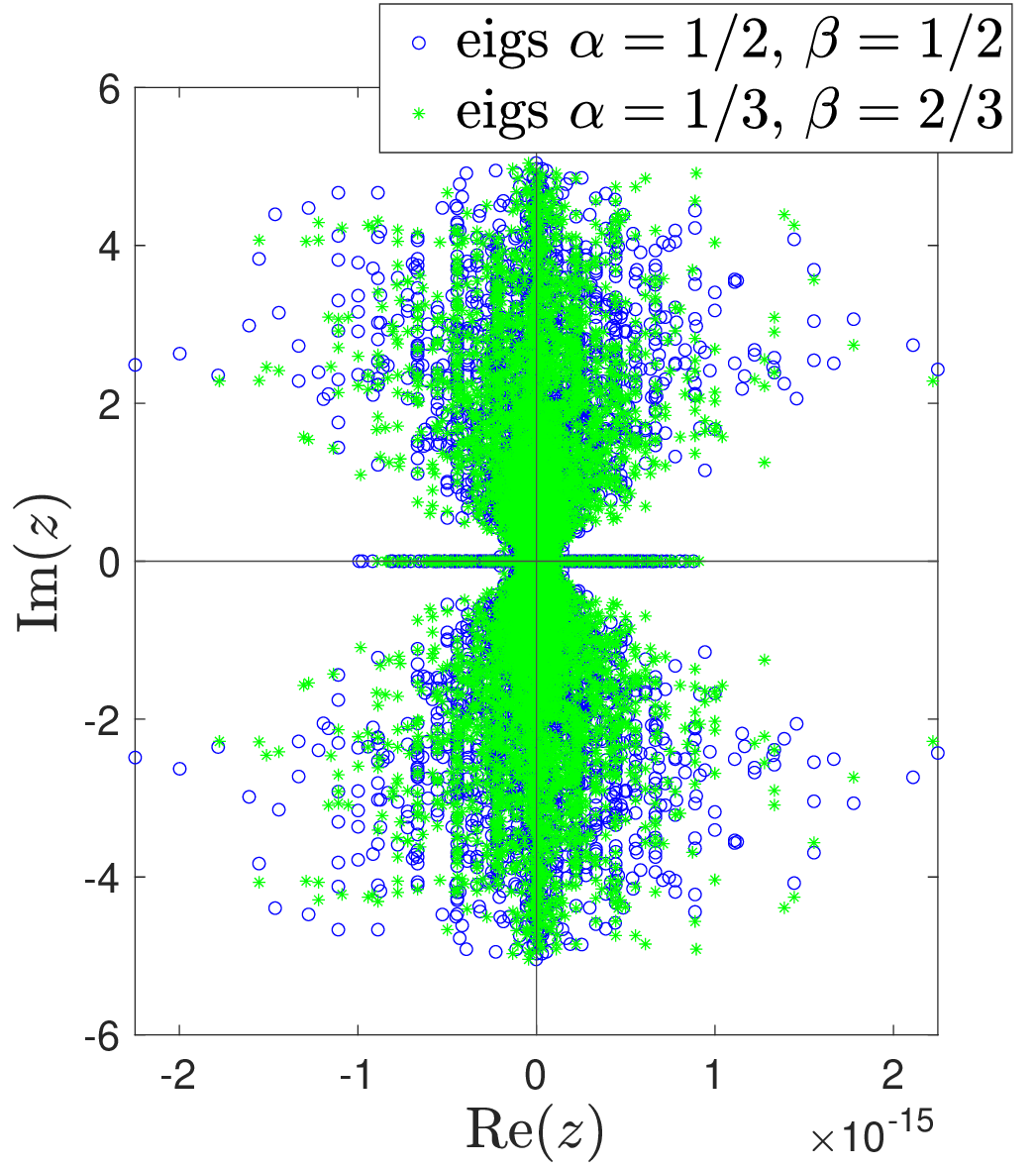}
    \end{subfigure}
    \caption{Spectrum of the eigenvalue problem  Eq.~\eqref{eqn:eigen_probl} for different values of the parameter $\alpha$ and $\beta$. On the right, the real part of the eigenvalues are equal to zero up to machine precision.}
    \label{fig:model_spectr}
\end{figure}
Fig.~\ref{fig:model_spectr} shows the spectrum's shape with two different sets of values of the parameters $\alpha$ and $\beta$. It is possible to observe that if the parameters do not satisfy the condition stated in Theorem~\ref{theo:energy}, for example, if $\alpha=\beta=1$ (red stars), eigenvalues with a real positive part appear, which implies that the energy of the system exhibits exponential growth. This instability is purely numerical, since we proved in Sec.~\ref{sec:ener_model} that energy is conserved at the continuous level. 
When $\alpha+\beta=1$, e.g. $\alpha=\beta=1/2$ (blue circles), the eigenvalues are purely imaginary, as expected from Theorem~\ref{theo:energy}. It is also possible to notice in the right plot in Fig.~\ref{fig:model_spectr} that other choices of $\alpha+\beta=1$ do not influence the spectrum of the eigenvalues.

\section{Shear Alfvén waves (SAWs)}
\label{sec:saws}
The SAWs are transverse anisotropic electromagnetic waves propagating in a magnetized plasma, \cite{chen2021physics}, \cite{stasiewicz2000small}. These waves are a stable perturbation of the electric and magnetic fields that are oriented perpendicular to each other, characterized by high frequencies with respect to the typical time scale of plasma turbulent phenomena. Indeed, when dealing with fluid simulation, \cite{zeiler1997nonlinear}, and gyrokinetic electron turbulence simulations, SAWs can impose severe limitations on the time step, affecting both the computational efficiency and stability of the simulation \cite{Jolliet:207684}.
In the subsequent sections, we employ the definitions of the parallel gradient and parallel Laplacian operator as introduced in previous sections. We assume $\nabla \cdot \textbf{b}=0$ and $\partial_t \textbf{b}=0$, being $\mathbf{b}$ the unit vector of the magnetic field in the SAWs context.

The electrostatic SAWs are described by the following system of equations, \cite{stasiewicz2000small}, \cite{Jolliet:207684}, which govern the evolution of the $\phi$ electrostatic potential and the $\vpare$ electron parallel velocity:
 \begin{equation}
\begin{alignedat}{2}
		 \pdv{ (\nlaplperp \phi)}{t} &= -\ngradpar V_{\parallel e},\\
		 \pdv{\vpare}{t} &= \zeta \ngradpar \phi +\eta \nlaplpar \vpare,
 	\end{alignedat}
  \label{eqn:saws}
\end{equation}
where $\nlaplperp$ is the perpendicular Laplacian operator, defined as $\nlaplperp: \mathbb{R}\to \mathbb{R}$ such that $\nlaplperp f= \nabla \cdot \left[(\mathbf{b}\times\boldsymbol{\nabla} f)\times \mathbf{b}\right]$. The ratio of ion and electron mass is $\zeta= \frac{m_i}{m_e}\gg1$ and $\eta \lesssim 1$. We consider homogeneous Dirichlet boundary conditions for all the variables.

We analyze in detail the dispersion relation of the electrostatic SAWs as in Eqs.~\eqref{eqn:saws}. In order to do that, we assume a perturbation of the form $\exp\left[i\left(k_y y+ k_z z-\omega t\right)\right]$ with respect to a mean field $\vpare^0$ for $\vpare$ and $\phi^0$ for $\phi$, considering the magnetic field almost parallel to the $z$ direction, as  \cite{Jolliet:207684}:
\begin{equation}
\begin{alignedat}{2}
		i\omega k_{\perp}^2 \phi^0 &= -ik_{\parallel} V_{\parallel e}^0,\\
		-i\omega\vpare^0 &= \zeta i k_{\parallel}\phi^0 -\eta k_{\parallel}^2\vpare^0,
\end{alignedat}
\end{equation}
where $k_{\parallel}$ is in the $z$ direction and $k_{\perp}$ in the $y$ direction.
Simplifying, we get
\begin{equation}
    \omega^2+2i\gamma_D\omega-\omega_0^2=0,
    \label{eqn:dispersion}
\end{equation}
whose solutions for $\omega$ are
\begin{equation}
    \omega=-2i\gamma_D \pm \sqrt{\omega_0^2-\gamma_D^2},
    \label{eqn:frequency}
\end{equation}
where 
$\omega_0=\sqrt{\zeta}k_{\parallel}/k_{\perp}$ and 
$\gamma_D=\frac{\eta}{2}k_{\parallel}^2$. The real part of Eq.~\eqref{eqn:frequency} gives the SAWs oscillation frequency, while the imaginary part gives its damping rate. The parallel diffusion introduces a damping rate proportional to $\gamma_D$ and decreases the frequency from $\omega_0$ to $\sqrt{\omega_0^2-\gamma_D^2}$. 
The oscillation becomes purely damped when $\gamma_D >\omega_0$ so when $\eta>2\sqrt{\zeta}/\left(k_{\parallel}k_{\perp}\right)$. 
In principle, for small values of $\eta$, increasing the number of planes in the $z$ direction increases the possible frequencies of the SAWs. As a consequence, the required time step to accurately capture the wave behavior should decrease. However, high $k_{\parallel}$ modes can be stabilized by adding parallel diffusion (increasing the damping rate).
On the other hand, the frequency of the SAWs increases as the system size increases (leading to a decrease in the smallest $k_\perp$ value), as shown in \cite{stegmeir2023analysis}. This phenomenon can adversely affect the allowed time step size in simulations, particularly as the size of the fusion device increases.

\subsection{Energy conservation}
\label{sec:energy_cons}
Considering the system of the SAWs with parallel diffusion reported in Eqs.~\eqref{eqn:saws}, we multiply by $\phi$ and by $\vpare$ the first and second equation respectively, integrate in space both equations over $\Omega$ and summing the two equations, similarly to the steps in Sec.~\ref{sec:energy_cons}, we obtain:
\begin{equation}
\begin{alignedat}{2}
         \int_{\Omega}\frac{1}{2}\frac{\partial }{\partial t} \left(\frac{1}{\zeta} \vpare^2+(\boldsymbol{\ngradperp} \phi)^2\right) \ d\Omega=& \int_{\Omega} \ngradpar \left( \phi \vpare \right)d \Omega - \\&\frac{\eta}{\zeta} \int_{\Omega} \boldsymbol{\nabla} \vpare \cdot \left(\textbf{b}\textbf{b}^T \boldsymbol{\nabla}\vpare \right)d\Omega.
     \label{eqn:energy_conse_saw}
     \end{alignedat}
\end{equation}
The inclusion of the parallel Laplacian in the equations induces a dissipative effect on the energy, which arises because the matrix $\textbf{b}\textbf{b}^T$ is symmetric positive semi--definite.
We observe that the first term on the right--hand side of Eq.~\eqref{eqn:energy_conse_saw} is the same as the one we found in the energy conservation of the wave model problem, Eq.~\eqref{eqn:conservation_model}. Consequently, Theorem~\ref{theo:energy} remains applicable in this context, with the exception that the energy in this system $E_s$ is defined differently and is dissipated due to the presence of the parallel Laplacian on the right--hand side.
Considering that homogeneous Dirichlet BCs are applied to both fields, we prove that the energy of the SAWs with parallel diffusion is dissipated in time; that is 
\begin{equation}
    \frac{\partial E_s }{\partial t}=\frac{\partial }{\partial t}\int_{\Omega}\frac{1}{2} \left(\frac{1}{\zeta} \vpare^2+(\boldsymbol{\ngradperp} \phi)^2\right) \ d\Omega\leq0.
    \label{eqn:energy_conservation}
\end{equation}

\subsection{Modeling the magnetic field for a Tokamak Configuration}
An important challenge from a numerical modeling point of view is the fact that the space scale of the phenomena happening in the direction parallel to the magnetic field is much longer compared to the one in the perpendicular direction. To handle the strong anisotropy between the parallel and perpendicular direction to the equilibrium magnetic field, an important characteristic in modeling the plasma dynamics, we introduce a magnetic field that is almost parallel to the $z$ coordinate and has small components along the $x$ and $y$ coordinates. 
Hence, we define the magnetic field $\textbf{b}$ as
\begin{equation}
    \Tilde{\textbf{b}}=-\textbf{e}_z-\varepsilon(\textbf{e}_z \cross \nabla \Psi),
    \label{eqn:saws_mg}
\end{equation}
where $\textbf{e}_z$ is the unit vector in the direction of the $z$ axis and $\Psi$ is a flux function with the form defined in Eq.~\eqref{eqn:psi} and $\varepsilon\ll 1$. It is important to notice that the above definition of the magnetic field satisfies Gauss's law. We take the function $\Psi$ to be similar to the shape of the magnetic field in a tokamak device with a lower--single null divertor configuration, as shown in Fig.~\ref{fig:quiver} and in Fig.~\ref{fig:contour} following Eq.~\eqref{eqn:saws_mg}. 
\begin{figure}[t]
    \centering
    \includegraphics[scale=0.37]{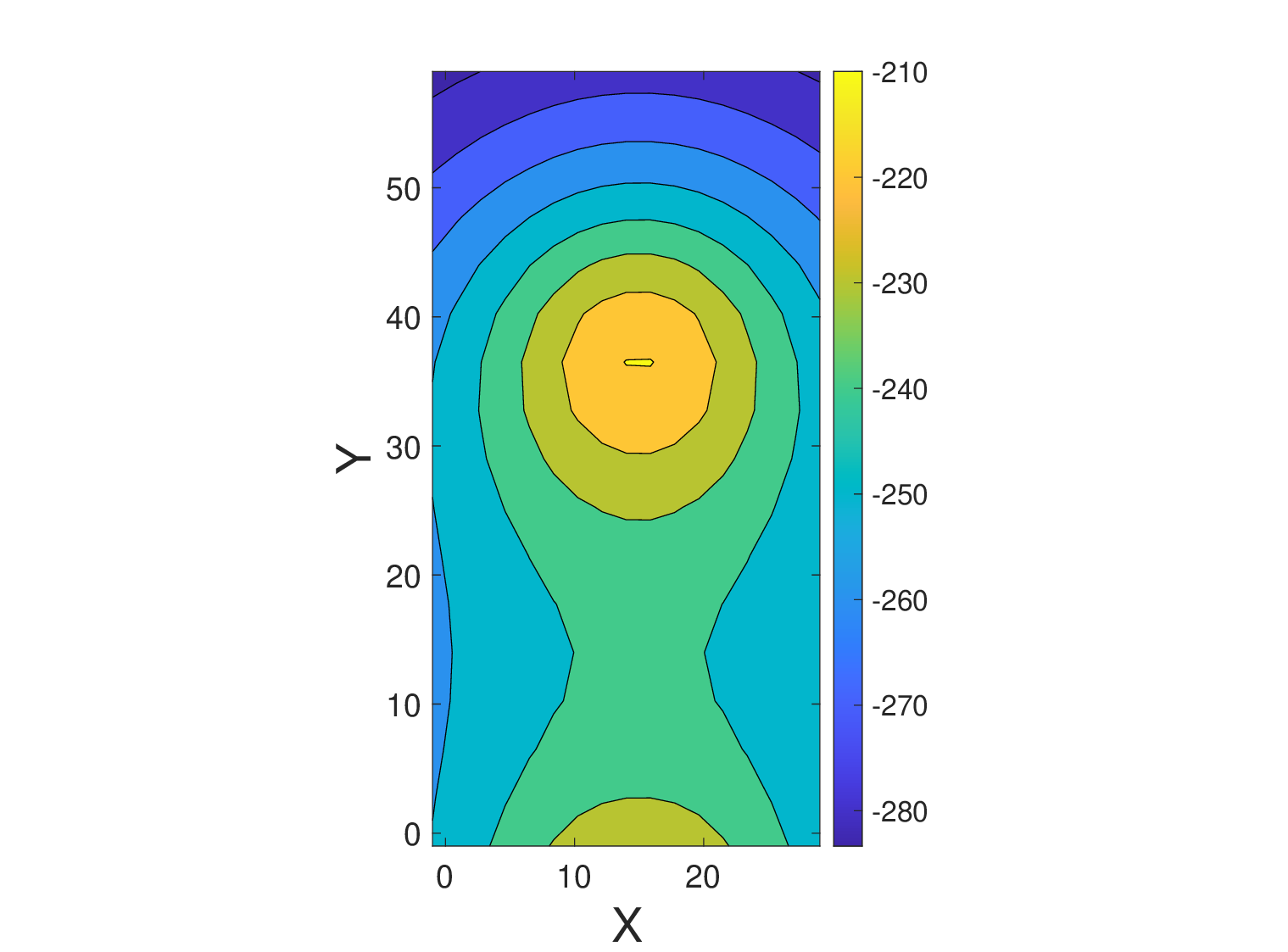}
    \caption{The contour plot of the function $\Psi$ in Eq.~\eqref{eqn:saws_mg}}
    \label{fig:contour}
\end{figure}
Moreover, the dynamics along the direction of the magnetic field are much faster compared to the phenomena happening in the perpendicular direction. As a result, the $x$ and $y$ coordinates are normalized using the ratio of the characteristic lengths $\varepsilon = h_{\perp} / h_{\parallel}$, where $h_{\parallel} \gg h_{\perp}$. Consequently, the dimensionless form of the gradient is defined as
\begin{equation}
    \widetilde{\nabla} f=\left[\frac{1}{\varepsilon
}\pdv{f}{x},\frac{1}{\varepsilon
}\pdv{f}{y},\pdv{f}{z}\right]^T
\end{equation}
and the definition of the parallel gradient is
\begin{equation}
    \widetilde{\nabla}_\parallel f=\textbf{b}\cdot \widetilde{\nabla}f,
    \label{eqn:saws_ngradpar}
\end{equation}
where $\textbf{b}$ is defined in Eq.~\eqref{eqn:saws_mg}.
Due to these scaling factors, the definition of the perpendicular Laplacian is equivalent to
\begin{equation}
\Tilde{\nlaplperp} f= \frac{\partial^2 f}{\partial x^2}+\frac{\partial^2 f}{\partial y^2} + \mathcal{O}(\varepsilon).
\label{eqn:perp_grad2}
\end{equation}
It is important to note that all calculations reported in Sec.~\ref{sec:energy_cons}, and consequently the final energy balance in Eq.~\eqref{eqn:energy_conservation}, can be performed using either the exact definition of the perpendicular Laplacian in Sec.~\ref{sec:saws} or the definition given in Eq.~\eqref{eqn:perp_grad2}, neglecting  $\mathcal{O}(\epsilon)$ and incurring an error proportional to $\mathcal{O}(\epsilon^2) $ in the energy computation.
In the following sections, we use dimensionless definitions for the operators neglecting the terms of order greater or equal than $\epsilon$ and we omit the tilde symbol for simplicity and readability: $\nlaplperp f= \partial_x^2f+ \partial_y^2f$, $\ngradpar f=\partial_y \Psi \partial_x f-\partial_x \Psi \partial_y f -\partial_z f$.

\subsection{Matrix formulation and discretization of the SAWs}
Starting from  Eqs.~\eqref{eqn:saws} and considering homogeneous Dirichlet boundary conditions for both field, we can write the matrix formulation of this system as
\begin{equation}
    \frac{d}{dt}
\underbrace{\left[
\begin{array}{c|c}
-\nlaplperp & \textbf{0} \\
\hline
\textbf{0} & \textbf{I}  \\
\end{array}
\right]}_{\mathbf{M}}\left[\begin{array}{c}
\Vec{\phi} \\
\Vec{\vpare}
\end{array}\right]
=\underbrace{\left[\begin{array}{c|c}
\textbf{0} & \ngradpar\vert_{pq} \\
\hline
\zeta\ngradpar\vert_{qp} & \eta\ngradpar^2  
\end{array}
\right]}_{\mathbf{D}}\left[\begin{array}{c}
\Vec{\phi} \\
\Vec{\vpare}
\end{array}\right].
\end{equation}
Since the mass matrix $\mathbf{M}$ is constant in time, we can write 
\begin{equation} 
\left[
\begin{array}{c|c}
-\ngradperp^2 & \textbf{0} \\
\hline
\textbf{0} & \textbf{I}  \\
\end{array}
\right]\frac{d}{dt}\left[\begin{array}{c}
\Vec{\phi} \\
\Vec{\vpare}
\end{array}\right]
=\left[\begin{array}{c|c}
\textbf{0} & \ngradpar\vert_{pq} \\
\hline
\zeta\ngradpar\vert_{qp} & \eta\ngradpar^2  
\end{array}
\right]\left[\begin{array}{c}
\Vec{\phi} \\
\Vec{\vpare}
\end{array}\right].
\label{eqn:algebraic_formulation}
\end{equation}
We choose to evaluate $\vpare$ on the $q$-grid and the $\phi$ on the $p$-grid. The operators $\ngradpar|_{pq}$ and $\ngradpar|_{qp}$ are discretized according to Eqs.~\eqref{eqn:weigter_ngradpar_pq} and \eqref{eqn:weigter_ngradpar_qp}. For the discretization of $\nlaplpar$, we adopt the second-order accurate symmetric scheme proposed in \cite{GUNTER2005354}, which employs a $3 \times 3 \times 3$ box stencil in three dimensions. To evaluate the parallel Laplacian of $\vpare$ at interior points located near the boundary, star--shaped $q$-grid points are introduced on the physical boundary (red line), as shown in Fig.~\ref{fig:ghosts}.
The scheme reported in \cite{GUNTER2005354} discretizes the parallel Laplacian as $\nabla \cdot (\textbf{b}\textbf{b}^T \nabla)$ with central finite differences: 
the gradient on the $q$-grid is first computed between adjacent points, yielding values on the $p$-grid where $\textbf{b}\textbf{b}^T$ is evaluated; central finite differences are then applied to discretize the divergence, mapping the result back onto the original $q$-grid. 
An adjustment is required for interior $q$-grid points near the boundary to maintain accuracy: in these cases, as shown in Fig.~\ref{fig:ghosts}, the adjacent points lying on the boundary are separated by half the usual grid spacing. Consequently, the scheme no longer relies on central finite differences, but instead accounts for the non-uniform spacing between points.
\begin{figure}[t]
    \centering   \includegraphics[scale=0.5]{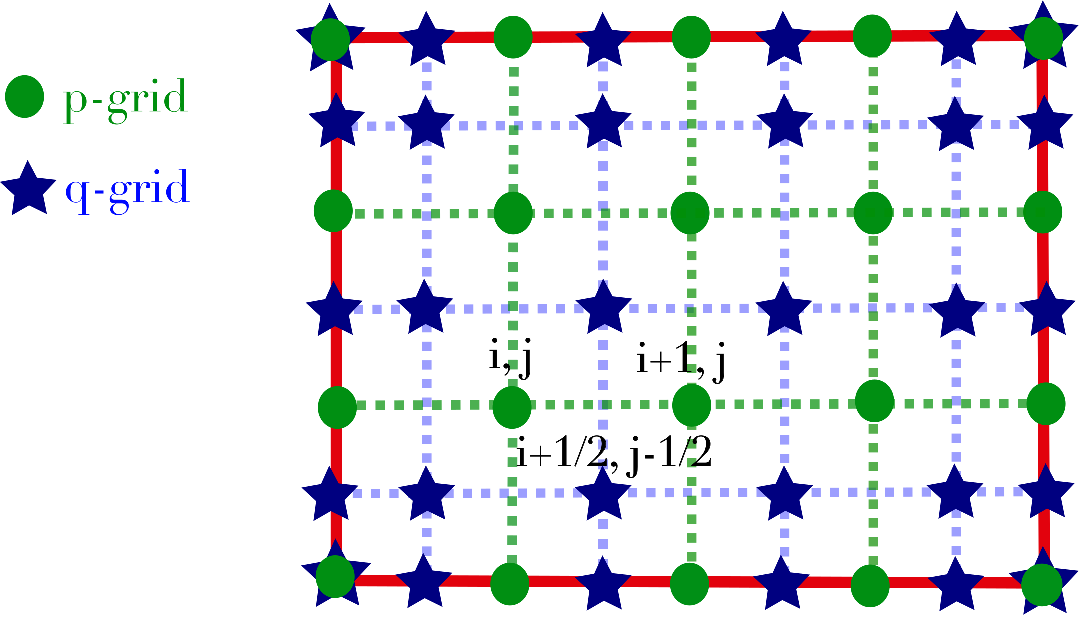}
    \caption{Sketch of the two grids in a two dimensional setting where the red line represents the physical domain, the green dashed line the $p$-grid with the green points as nodes and the blue dashed line describes the $q$-grid with blue stars as nodes. The addition of $q$ nodes on the physical boundary is necessary to compute the parallel Laplacian operator.}
    \label{fig:ghosts}
\end{figure}
The perpendicular Laplacian, expressed as the sum of second--order derivatives in the $x$ and $y$ directions, is discretized using second--order accurate central finite difference schemes. Both the parallel and perpendicular Laplacian operators are defined such that their input and output reside on the same computational grid.

\subsection{Numerical Simulation of SAWs system}
To find the eigenvalues of this system, we solve the generalized eigenvalue problem 
\begin{equation}
    \textbf{D}\Vec{v}=\lambda \textbf{M}\Vec{v} \quad \text{with } \Vec{v}=\left[\begin{array}{c}
\Vec{\phi} \\
\Vec{\vpare}
\end{array}\right],
    \label{eqn:generalized}
\end{equation}
where $\Vec{\phi}$ and $\Vec{\vpare}$ are the vectors associated with the scalar field following the convention introduced in Eq.~\eqref{eqn:vecform} and $\textbf{M}$ is a symmetric positive definitive matrix, considering the application of the boundary conditions.
In this numerical test, we consider $\zeta=2500$ and $\eta=\frac{4}{3}$.
By solving the generalized eigenvalue problem given by Eq.~\eqref{eqn:generalized}, to estimate the stability of the differential algebraic equation, we obtain the spectrum displayed in Fig.~\ref{fig:spectrumold}.
As in the previous cases, when the values of $\alpha$ and $\beta$ satisfy Theorem~\ref{theo:energy}, the spectrum of the system presents only eigenvalues with real negative part. 
Despite the presence of diffusion in the system described by Eqs.~\eqref{eqn:saws}, the module of the imaginary component of the eigenvalues is notably larger than the real component. This indicates that when using explicit time discretizations, the constraint on the time step is predominantly due to the imaginary part of the eigenvalues. 
 \begin{figure}[t]
     \centering
     \begin{subfigure}{0.56\textwidth}
        \includegraphics[scale=0.3]{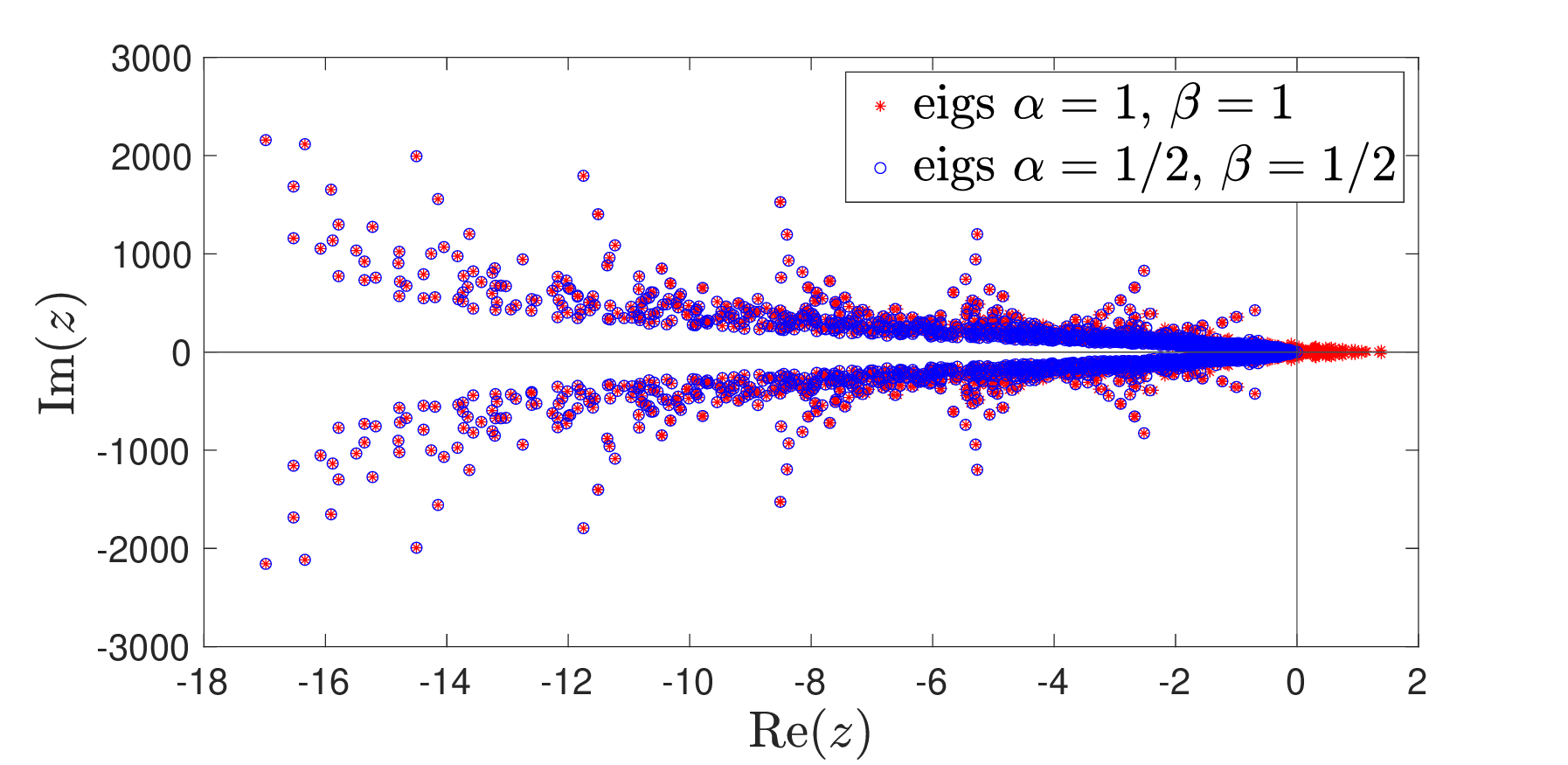}
     \end{subfigure}
     \begin{subfigure}{0.43\textwidth}
         \includegraphics[scale=0.3]{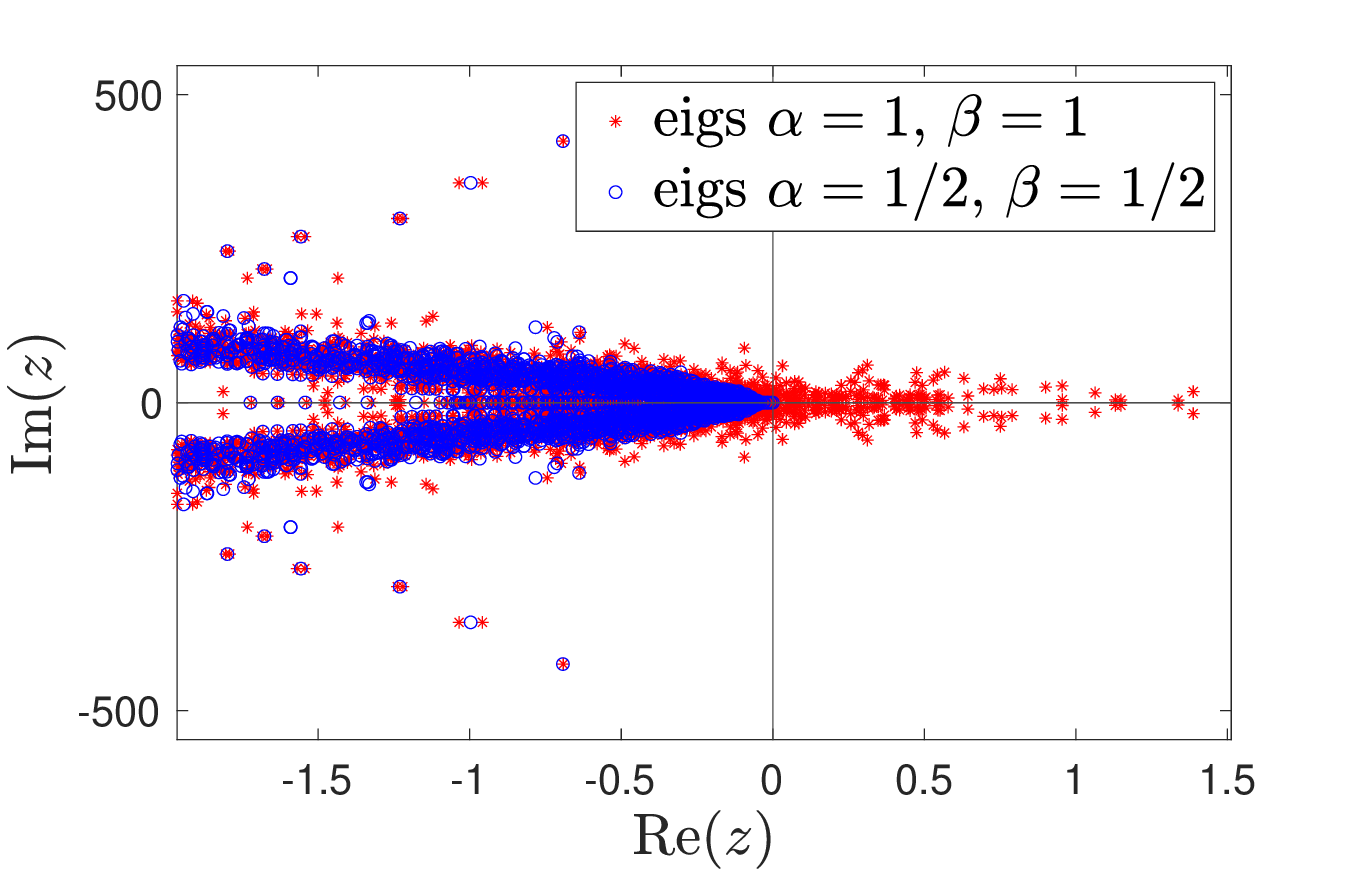}
     \end{subfigure}
     \caption{Spectrum of the generalized eigenvalue problem Eq.~\eqref{eqn:generalized}, for different values of $\alpha$ and $\beta$.}
     \label{fig:spectrumold}
 \end{figure}
\begin{figure}[t]
  \begin{subfigure}{0.48\textwidth}
    \centering
    \includegraphics[width=\textwidth]{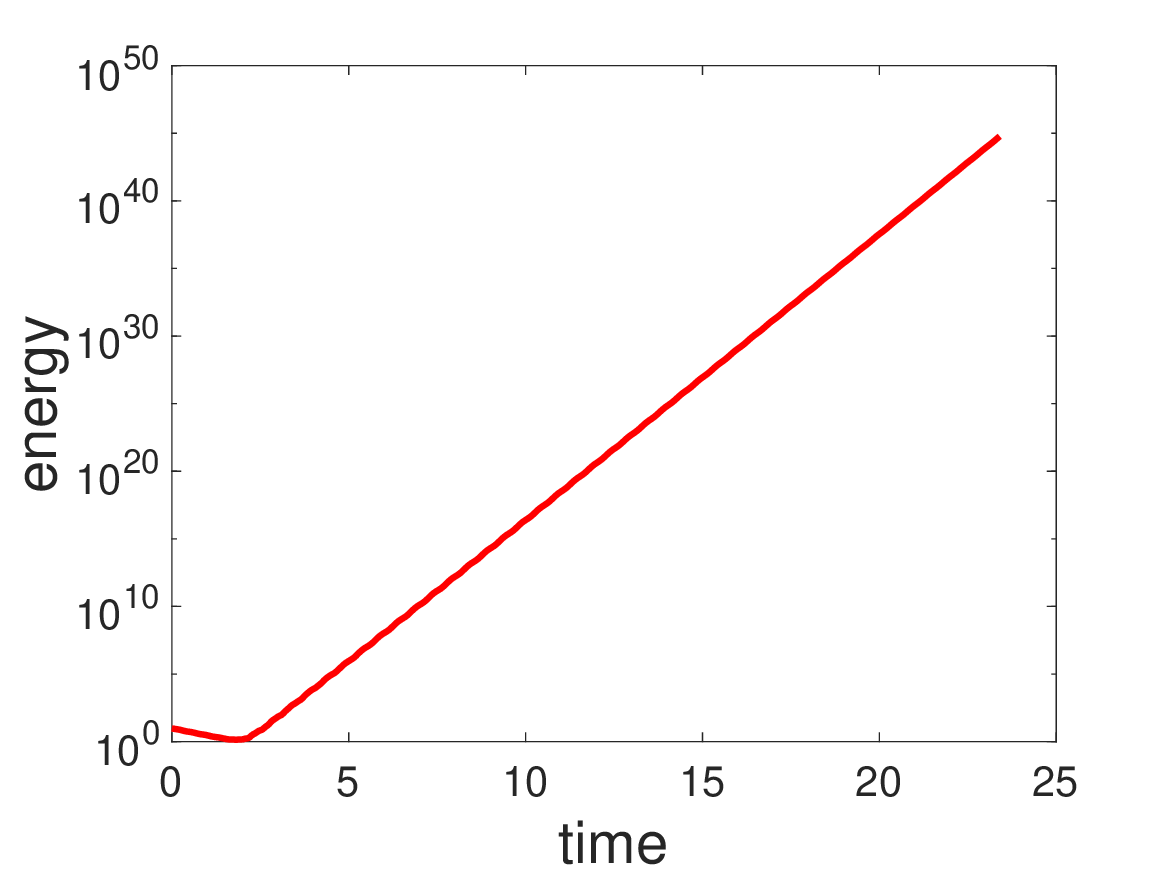}
    \caption{}
    \label{fig:energy_old}
    \end{subfigure}
    \begin{subfigure}{0.48\textwidth}
        \centering
    \includegraphics[width=\textwidth]{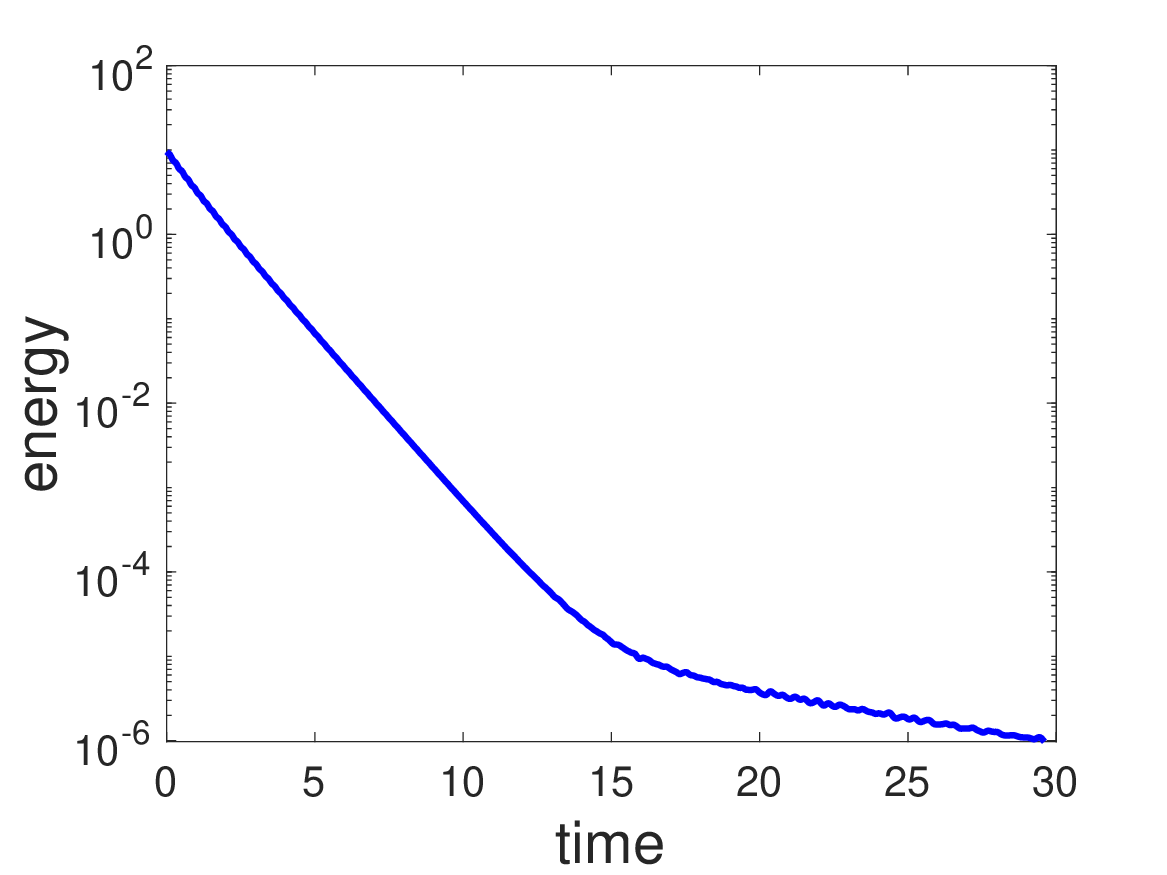}
    \caption{}
    \label{fig:energy_new}
    \end{subfigure}
    \caption{Semi--logarithmic plot of the energy versus time. The discretization used is $N_x\times N_y\times N_z=16\times 16\times16$. Fig.~\ref{fig:energy_old} shows the energy profile using $\alpha=1$ and $\beta=1$ in the formula of the parallel gradient operator instead Fig.~\ref{fig:energy_new} shows the profile of the energy when $\alpha=1/2$ and $\beta=1/2$}
    \label{fig:energy}
\end{figure}

We analyze the evolution of the energy profile for different values of $\alpha$ and $\beta$, as reported in Fig.~\ref{fig:energy}. 

Since the system presented in Eqs.~\eqref{eqn:saws} includes a diffusive term on the right--hand side and, therefore the energy is dissipated in time, we employ a standard fourth--order Runge–Kutta (RK4) method for time integration. To ensure numerical stability, we select a time step $\Delta t = 10^{-3}$ such that all eigenvalues with a negative real part lie within the RK4 stability region.

We note that in scenarios where the parallel diffusion term is absent and energy conservation is critical, the Crank--Nicolson method offers a viable alternative due to its unconditional stability and its ability to preserve a discretized version of the system’s energy, as shown in \ref{app2}. However, the Crank--Nicolson algorithm is known to introduce spurious oscillations, particularly when applied to problems with sharp gradients or discontinuities, \cite{osterby2003five}. These oscillations arise from the method's dispersive properties, leading to multiple modes propagating at different speeds. To mitigate this issue while preserving wave characteristics, low--dissipation and low--dispersion Runge–Kutta (LDLDRK) schemes, such as those proposed by \cite{hu1996low}, can be considered as alternatives to the classical RK4 method.
It is evident from Fig.~\ref{fig:energy_old} that when employing values of $\alpha$ and $\beta$ not satisfying Theorem~\ref{theo:energy}, the energy value exhibits exponential growth in time. In contrast, the energy exponentially decreases to zero with the energy--preserving implementation of the parallel gradient, as shown in Fig.~\ref{fig:energy_new}. In this figure, we present a specific case with $\alpha=1/2$ and $\beta=1/2$. However, we observe the same behavior for all values of $\alpha$ and $\beta$ satisfying Theorem~\ref{theo:energy}.
\begin{figure}
    \centering
    \includegraphics[width=0.8\linewidth]{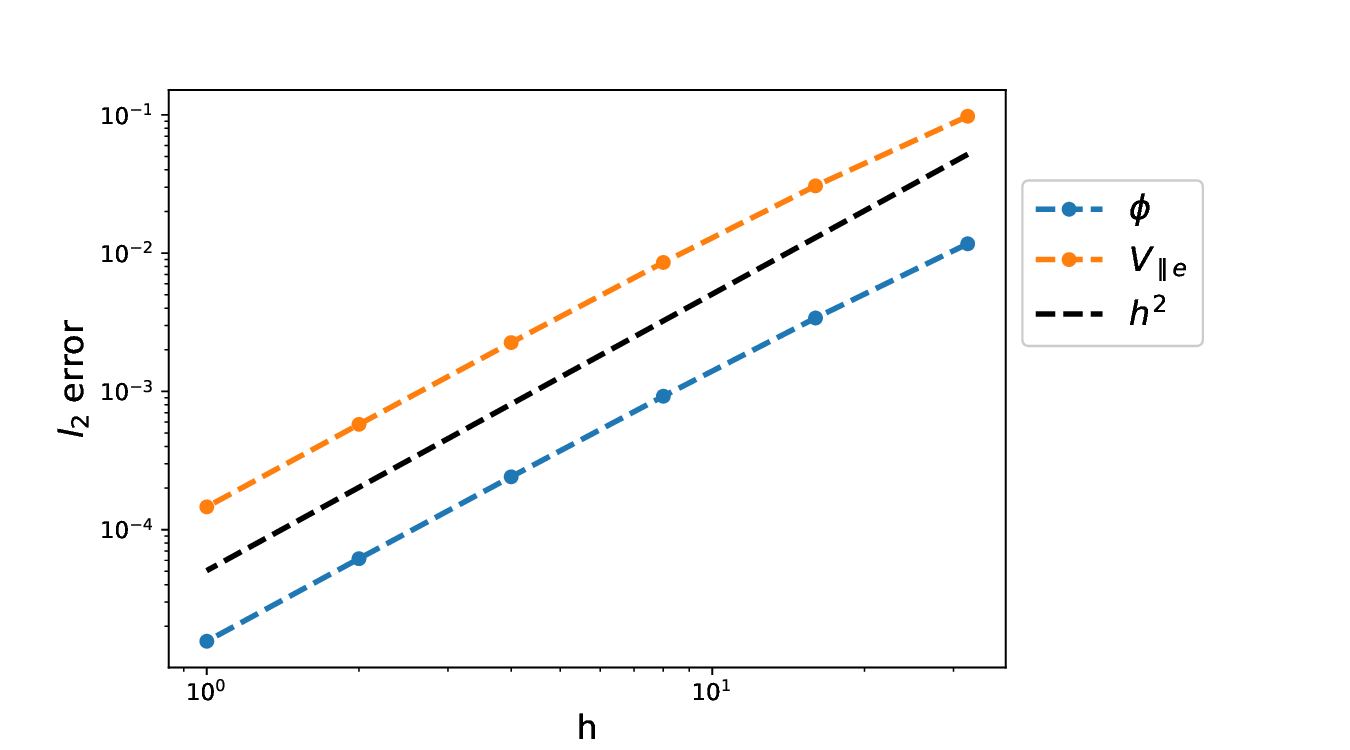}
    \caption{Log--log plot of the $L_2$ error convergence for the variables $\phi$ and $V_{\parallel e}$. The test is performed using the advective field $\textbf{b}$ as defined in Eq.~\eqref{eqn:psi}, with uniform spatial resolution such that $h = \frac{\Delta x}{\Delta x_0} \times \frac{\Delta y}{\Delta y_0}\times \frac{\Delta z}{\Delta z_0}$ where $\Delta x_0$, $\Delta y_0$ and $\Delta z_0$ are the grid spacings for $N_x\times N_y\times N_y= 256\times 256\times 256$. In the $\ngradpar |_{pq}$ and $\ngradpar |_{qp}$, $\alpha$ and $\beta$ are chosen to be equal to $1/2$.}
    \label{fig:convergence_saw}
\end{figure}

Finally, we perform a numerical test to evaluate the order of accuracy of the solution using the Method of Manufactured Solutions (MMS) \cite{riva2014verification}. The imposed analytical solutions are provided in \ref{app1}. We investigate spatial convergence by fixing the time step and the final simulation time to $\Delta t=1.25 \times 10^{-5}$ and $T=0.002$, respectively. Since all differential operators are discretized using second--order accurate finite difference schemes, the discretization error is expected to decrease quadratically with the grid spacing. This expectation is confirmed by the results in Fig.~\ref{fig:convergence_saw}, which demonstrate second--order convergence in the $l_2$ norm for both $\phi $ and $\vpare$.

\section{Conclusions}
\label{sec:concl}
In this work, we introduce a novel mimetic finite difference (MFD) scheme for the advective term with divergence--free advective field, designed for staggered grids. The proposed discretization of the parallel gradient operator ensures that the divergence theorem is preserved at the discrete level, under the assumption of homogeneous Dirichlet boundary conditions. 

The method leverages the divergence--free property of the advective field, an intrinsic characteristic of the magnetic fields and divergence--free velocity fields, such as those encountered in the convective term of the Navier-Stokes equations. By exploiting this feature, the parallel gradient operator is reformulated as a weighted average of the advective operator $\mathbf{b}\cdot \boldsymbol{\nabla}\bullet$ and the divergence operator $\nabla \cdot \left(\mathbf{b}\bullet\right)$. This approach aligns with the skew--symmetric formulation proposed in \cite{morinishi2010skew}. 

At the continuous level, this reformulation is mathematically equivalent to the original equations, preserving the physical and mathematical properties of the system. However, at the discrete level, it introduces a significant advantage: it ensures energy conservation within the numerical scheme. This conservation enhances both the stability and physical fidelity of the simulations, allowing the discrete approximation to better reflect the underlying physics of the system.

The stability of the method is validated through its application to a wave--like model problem and a system representing the shear Alfvén waves (SAWs). This validation involves computing the spectrum of the generalized eigenvalue problem, following the methodology outlined in \cite{du2013robust}. The analysis demonstrates that the new discretization method successfully preserves energy, which has significant implications for improving the accuracy and robustness of various numerical simulations, particularly those used in fluid plasma codes. 

By ensuring energy conservation, the method not only enhances the physical fidelity of the simulations but also mitigates numerical artifacts, a crucial consideration in modeling complex plasma dynamics.
Future work may explore extending this approach to more complex boundary conditions and multi--dimensional systems, further enhancing its applicability and impact.

\section*{Data Availability Statement}
The data that support the findings of this study are available upon reasonable request from the authors.

\section*{Acknowledgements}
This work has been carried out within the framework of the EUROfusion Consortium, partially funded by the European Union via the Euratom Research and Training Programme (Grant Agreement No 101052200 — EUROfusion). The Swiss contribution to this work has been funded by the Swiss State Secretariat for Education, Research and Innovation (SERI). Views and opinions expressed are however those of the author(s) only and do not necessarily reflect those of the European Union, the European Commission, or SERI. Neither the European Union nor the European Commission nor SERI can be held responsible for them.
\appendix
\section{Parameters in the test cases}
\label{app1}
Here we report the parameters used to define $\Psi$:
\begin{align*}
    A_{mag}=&\frac{25L_x}{12}L_y\\
    x_{mag}=&\frac{L_x}{2}\\
    y_{mag1}=&\frac{5}{8}(y_{max}-y_{min})\\
    y_{mag2}=&18\left(\frac{y_{max}-y_{min}}{40}\right)-y_{mag1}\\
    a_s=&\frac{5}{40}(y_{max}-y_{min})&
\end{align*}
where $L_x$ and $L_y$ are the sizes of the domain in $x$ and $y$ that in our simulations are $L_x=30$ and $L_y=60$. 

Here we report the solutions imposed in the MMS to test the convergence of the variables in the  SAWs system:
\begin{equation*}
\begin{alignedat}{2}
    \phi^s(x,y,z)=&3.1+0.8\sin(y+0.1)\sin(z)\sin(x+t)\\
    \vpare^s(x,y,z)=&\sin(y+0.2)\sin(z+0.2)\sin(x+t+0.2)
\end{alignedat}
\end{equation*}

\section{Discretized energy conservation}
\label{app2}
The leap--frog method, a widely used multi--step approach for solving the wave equation, has second--order accuracy, and stability is not guaranteed with larger time steps. Energy--conserving methods, like theta methods, have been introduced to improve accuracy while maintaining conservation. On the other hand, the standard Crank--Nicolson method preserves conservation laws but only offers second--order accuracy.
If we apply the Crank--Nicolson scheme to the Eq.~\eqref{eqn:algebraic_formulation} without parallel diffusion, so considering a pure wave problem, we get:
\begin{equation*}
\begin{alignedat}{2}
    \Tilde{\textbf{M}}\frac{\Vec{v}^{n+1}-\Vec{v}^n}{\Delta t}=&\Tilde{\textbf{D}}\frac{\Vec{v}^{n+1}+\Vec{v}^{n}}{2} \\ \text{where } \Tilde{\textbf{M}}=\left[\begin{array}{c|c}
-\nlaplperp & \textbf{0} \\
\hline
\textbf{0} & \frac{1}{\zeta}\textbf{I}  \\
\end{array}
\right]\ &\text{ and }\Tilde{\textbf{D}}=\left[\begin{array}{c|c}
\textbf{0} & \ngradpar\vert_{pq} \\
\hline
\ngradpar\vert_{qp} & \textbf{0}
\end{array}
\right].
\end{alignedat}
\end{equation*}
It is possible to multiply the previous equation for $(\Vec{v}^{n+1}+\Vec{v}^{n})^T/2\mathbf{X}$ to obtain:
\begin{equation}
\begin{alignedat}{2}
\frac{(\Vec{v}^{n+1}+\Vec{v}^{n})^T}{2}\mathbf{X}\Tilde{\textbf{M}}\frac{\Vec{v}^{n+1}-\Vec{v}^p}{\Delta t}=&\frac{(\Vec{v}^{n+1}+\Vec{v}^{n})^T}{2}\mathbf{X}\Tilde{\textbf{D}}\frac{\Vec{v}^{n+1}+\Vec{v}^{n}}{2} \\ \text{where }\mathbf{X}=&\left[\begin{array}{c|c}
\textbf{X}_p & \textbf{0} \\
\hline
\textbf{0} & \textbf{X}_q
\end{array}
\right]
\end{alignedat}
\label{eqn:dd3}
\end{equation}
The rhs of the Eq.~\eqref{eqn:dd3} for the interior point (so we can neglect the matrix norm), vanishes since
\begin{equation}
\begin{alignedat}{2}
    \left[
\begin{array}{c}
\Vec{\phi}^{n+\frac{1}{2}} \\
\Vec{\vpare}^{w}
\end{array}\right]^T\Tilde{\textbf{D}}\left[
\begin{array}{c}
\Vec{\phi}^{n+\frac{1}{2}} \\
\Vec{\vpare}^{n+\frac{1}{2}}
\end{array}\right]=  
&(\Vec{\phi}^{n+\frac{1}{2}})^T\textbf{C}|_{pq}\Vec{\vpare}^{n+\frac{1}{2}} -(\Vec{\vpare}^{n+\frac{1}{2}})^T\textbf{C}|_{pq}^T \Vec{\phi}^{n+\frac{1}{2}}\\=&(\Vec{\phi}^{n+\frac{1}{2}})^T\textbf{C}|_{pq}\Vec{\vpare}^{n+\frac{1}{2}} -\left((\Vec{\phi}^{n+\frac{1}{2}})^T\textbf{C}|_{pq} \Vec{\vpare}^{n+\frac{1}{2}}\right)^T\\=&0,
\label{eqn:rhs_is0}
\end{alignedat}
\end{equation}
where $\Vec{\phi}^{n+\frac{1}{2}}=\left(\Vec{\phi}^{n+1}+ \Vec{\phi}^n\right)/2$ and $\Vec{\vpare}^{n+\frac{1}{2}}=\left(\Vec{\vpare}^{n+1}+\Vec{\vpare}^n\right)/2$. For the points at the boundary, the previous expression vanishes for the imposed boundary conditions.
The left--hand side express the discretized energy is conserved since $\textbf{C}|_{qp}=-\textbf{C}|_{pq}^T$:
\begin{equation}
\begin{alignedat}{2}
\left(\frac{\Vec{v}^{n+1}+\Vec{v}^{n}}{2}\right)^T\textbf{X}\Tilde{\textbf{M}}\frac{\Vec{v}^{n+1}-\Vec{v}^n}{\Delta t}=& \left(\frac{\Vec{\phi}^{n+1}+\Vec{\phi}^n}{2}\right)^T\textbf{X}_p\left(-\nlaplperp\left(\frac{\Vec{\phi}^{n+1}-\Vec{\phi}^n}{\Delta t}\right)\right)\\+& \left(\frac{\Vec{\vpare}^{n+1}+\Vec{\vpare}^n}{2}\right)^T\textbf{X}_q\left(\frac{\Vec{\vpare}^{n+1}-\Vec{\vpare}^n}{\Delta t}\right).
\end{alignedat}
\label{eqn:dd4}
\end{equation}
The energy conserved in the continuous setting, as reported in Eq.~\eqref{eqn:energy_conservation}, is equivalent to the following expression:
\begin{equation}
\frac{\partial E_s}{\partial t} = \int_{\Omega} \vpare \pdv{\vpare}{t} \  d\Omega -\int_{\Omega} \pdv{(\ngradperp^2 \phi)}{t} \phi  \ d\Omega.
\end{equation}
This represents the continuous form of the derivative in time of the quantity we conserve in the discrete settings, as shown in Eq.~\eqref{eqn:dd4}. In conclusion, by employing the Crank--Nicolson method for time integration alongside our MFD scheme for the parallel gradient operator, we successfully conserve the discrete version of the energy in time.

\bibliographystyle{unsrt}  
  \bibliography{biblio1}

\begin{thebibliography}{10}

\bibitem{patankar2018numerical}
Suhas Patankar.
\newblock {\em Numerical heat transfer and fluid flow}.
\newblock CRC press, 2018.

\bibitem{durran2010numerical}
Dale~R Durran.
\newblock {\em Numerical methods for fluid dynamics: With applications to geophysics}, volume~32.
\newblock Springer Science \& Business Media, 2010.

\bibitem{durran2013numerical}
Dale~R Durran.
\newblock {\em Numerical methods for wave equations in geophysical fluid dynamics}, volume~32.
\newblock Springer Science \& Business Media, 2013.

\bibitem{sharma2021introduction}
Atul Sharma.
\newblock {\em Introduction to computational fluid dynamics: development, application and analysis}.
\newblock Springer Nature, 2021.

\bibitem{yee1966numerical}
Kane Yee.
\newblock Numerical solution of initial boundary value problems involving maxwell's equations in isotropic media.
\newblock {\em IEEE Transactions on antennas and propagation}, 14(3):302--307, 1966.

\bibitem{lipnikov2014mimetic}
Konstantin Lipnikov, Gianmarco Manzini, and Mikhail Shashkov.
\newblock Mimetic finite difference method.
\newblock {\em Journal of Computational Physics}, 257:1163--1227, 2014.

\bibitem{arakawa1981potential}
Akio Arakawa and Vivian~R Lamb.
\newblock A potential enstrophy and energy conserving scheme for the shallow water equations.
\newblock {\em Monthly Weather Review}, 109(1):18--36, 1981.

\bibitem{arakawa1977computational}
Akio Arakawa and Vivian~R Lamb.
\newblock Computational design of the basic dynamical processes of the ucla general circulation model.
\newblock {\em General circulation models of the atmosphere}, 17(Supplement C):173--265, 1977.

\bibitem{fernandez2014review}
David C Del~Rey Fern{\'a}ndez, Jason~E Hicken, and David~W Zingg.
\newblock Review of summation-by-parts operators with simultaneous approximation terms for the numerical solution of partial differential equations.
\newblock {\em Computers \& Fluids}, 95:171--196, 2014.

\bibitem{svard2014review}
Magnus Sv{\"a}rd and Jan Nordstr{\"o}m.
\newblock Review of summation-by-parts schemes for initial--boundary-value problems.
\newblock {\em Journal of Computational Physics}, 268:17--38, 2014.

\bibitem{o2017energy}
Ossian O'Reilly, Tomas Lundquist, Eric~M Dunham, and Jan Nordstr{\"o}m.
\newblock Energy stable and high-order-accurate finite difference methods on staggered grids.
\newblock {\em Journal of Computational Physics}, 346:572--589, 2017.

\bibitem{garbet2010gyrokinetic}
Xavier Garbet, Yasuhiro Idomura, Laurent Villard, and TH~Watanabe.
\newblock Gyrokinetic simulations of turbulent transport.
\newblock {\em Nuclear Fusion}, 50(4):043002, 2010.

\bibitem{stegmeir2023analysis}
A~Stegmeir, T~Body, and W~Zholobenko.
\newblock Analysis of locally-aligned and non-aligned discretisation schemes for reactor-scale tokamak edge turbulence simulations.
\newblock {\em Computer Physics Communications}, 290:108801, 2023.

\bibitem{umansky2009status}
MV~Umansky, XQ~Xu, Ben Dudson, LL~LoDestro, and JR~Myra.
\newblock Status and verification of edge plasma turbulence code bout.
\newblock {\em Computer Physics Communications}, 180(6):887--903, 2009.

\bibitem{giacomin2022gbs}
Maurizio Giacomin, Paolo Ricci, A~Coroado, G~Fourestey, Davide Galassi, Emmanuel Lanti, D~Mancini, Nicolas Richart, Louis~N Stenger, and N~Varini.
\newblock The gbs code for the self-consistent simulation of plasma turbulence and kinetic neutral dynamics in the tokamak boundary.
\newblock {\em Journal of Computational Physics}, 463:111294, 2022.

\bibitem{ricci2012simulation}
P~Ricci, FD~Halpern, S~Jolliet, J~Loizu, A~Mosetto, A~Fasoli, I~Furno, and C~Theiler.
\newblock Simulation of plasma turbulence in scrape-off layer conditions: the gbs code, simulation results and code validation.
\newblock {\em Plasma Physics and Controlled Fusion}, 54(12):124047, 2012.

\bibitem{GUNTER2005354}
S.~Günter, Q.~Yu, J.~Krüger, and K.~Lackner.
\newblock Modelling of heat transport in magnetised plasmas using non-aligned coordinates.
\newblock {\em Journal of Computational Physics}, 209(1):354--370, 2005.

\bibitem{gunter2007finite}
Sibylle G{\"u}nter, Karl Lackner, and C~Tichmann.
\newblock Finite element and higher order difference formulations for modelling heat transport in magnetised plasmas.
\newblock {\em Journal of Computational Physics}, 226(2):2306--2316, 2007.

\bibitem{morinishi2010skew}
Yohei Morinishi.
\newblock Skew-symmetric form of convective terms and fully conservative finite difference schemes for variable density low-mach number flows.
\newblock {\em Journal of Computational Physics}, 229(2):276--300, 2010.

\bibitem{morinishi1998fully}
Yohei Morinishi, Thomas~S Lund, Oleg~V Vasilyev, and Parviz Moin.
\newblock Fully conservative higher order finite difference schemes for incompressible flow.
\newblock {\em Journal of computational physics}, 143(1):90--124, 1998.

\bibitem{halpern2018anti}
Federico~D Halpern and Ronald~E Waltz.
\newblock Anti-symmetric plasma moment equations with conservative discrete counterparts.
\newblock {\em Physics of Plasmas}, 25(6), 2018.

\bibitem{halpern2020anti}
Federico~D Halpern.
\newblock Anti-symmetric representation of the extended magnetohydrodynamic equations.
\newblock {\em Physics of Plasmas}, 27(4), 2020.

\bibitem{halpern2024paralleldiffusionoperatormagnetized}
Federico~D. Halpern, Min-Gu Yoo, Brendan Lyons, and Juan~Diego Colmenares.
\newblock Parallel diffusion operator for magnetized plasmas with improved spectral fidelity, 2024.

\bibitem{Jolliet:207684}
S.~Jolliet, F.D. Halpern, J.~Loizu, A.~Mosetto, F.~Riva, and P.~Ricci.
\newblock Numerical approach to the parallel gradient operator in tokamak scrape-off layer turbulence simulations and application to the gbs code.
\newblock {\em Computer Physics Communications}, 188:21--32, 2015.

\bibitem{du2013robust}
Nguyen~Huu Du, Vu~Hoang Linh, and Volker Mehrmann.
\newblock Robust stability of differential-algebraic equations.
\newblock {\em Surveys in Differential-Algebraic Equations I}, pages 63--95, 2013.

\bibitem{chen2021physics}
Liu Chen, Fulvio Zonca, and Yu~Lin.
\newblock Physics of kinetic alfv{\'e}n waves: a gyrokinetic theory approach.
\newblock {\em Reviews of Modern Plasma Physics}, 5:1--37, 2021.

\bibitem{stasiewicz2000small}
K~Stasiewicz, P~Bellan, C~Chaston, C~Kletzing, R~Lysak, J~Maggs, O~Pokhotelov, C~Seyler, P~Shukla, L~Stenflo, et~al.
\newblock Small scale alfv{\'e}nic structure in the aurora.
\newblock {\em Space Science Reviews}, 92:423--533, 2000.

\bibitem{zeiler1997nonlinear}
A~Zeiler, JF~Drake, and B~Rogers.
\newblock Nonlinear reduced braginskii equations with ion thermal dynamics in toroidal plasma.
\newblock {\em Physics of Plasmas}, 4(6):2134--2138, 1997.

\bibitem{osterby2003five}
Ole {\O}sterby.
\newblock Five ways of reducing the crank--nicolson oscillations.
\newblock {\em BIT Numerical Mathematics}, 43:811--822, 2003.

\bibitem{hu1996low}
FQ~Hu, M~Yousuff Hussaini, and JL~Manthey.
\newblock Low-dissipation and low-dispersion runge--kutta schemes for computational acoustics.
\newblock {\em Journal of computational physics}, 124(1):177--191, 1996.

\bibitem{riva2014verification}
Fabio Riva, Paolo Ricci, Federico~D Halpern, S{\'e}bastien Jolliet, Joaquim Loizu, and Annamaria Mosetto.
\newblock Verification methodology for plasma simulations and application to a scrape-off layer turbulence code.
\newblock {\em Physics of Plasmas}, 21(6), 2014.

\end{thebibliography}






\end{document}